\documentclass[11pt]{amsart}
\usepackage[utf8]{inputenc}
\usepackage{mathptmx} 
\usepackage{caption}
\usepackage[scaled=0.90]{helvet} 
\usepackage{courier} 
\usepackage[T1]{fontenc}

\usepackage{float}
\usepackage[pdftex,bookmarks,colorlinks,breaklinks]{hyperref} 
\usepackage{amsmath}
\usepackage{wrapfig}
\usepackage{amsthm}
\usepackage{amsfonts}
\usepackage{amssymb}
\usepackage{graphicx}

\usepackage{amssymb,amsmath,amsthm}
\usepackage{mathrsfs}
\usepackage[margin=0.96in]{geometry}
\usepackage{color}

\newtheorem{Th}{Theorem}

\newtheorem{Le}{Lemma}
\newtheorem{Def}{Definition}
\newtheorem{Cor}{Corollary}

\newtheorem{Rem}{Remark}
\newtheorem{Prop}{Proposition}

\usepackage{graphicx}
\newcommand{\hreff}[1]{\hyperref[#1]{\ref{#1}}}

\DeclareMathOperator{\sign}{sign}

\newcommand{\av}[2]{\langle {#1}\rangle_{{}_{#2}}}
\newcommand{\Bell}{\mathbf{B}}

\newcommand{\df}{\buildrel\mathrm{def}\over=}
\newcommand{\s}{\mathbf}
\newcommand{\eps}{\varepsilon}

\begin{document}

\title{Inequality for Burkholder's martingale transform}

\author{Paata Ivanisvili}
\address{Department of Mathematics,  Michigan State University, East
Lansing, MI 48824, USA}
\email{ivanisvi@math.msu.edu}
\urladdr{http://math.msu.edu/~ivanisvi}

\makeatletter
\@namedef{subjclassname@2010}{
  \textup{2010} Mathematics Subject Classification}
\makeatother




%
%


   \begin{abstract}
We find the sharp constant $C=C(\tau,p, \mathbb{E}G, \mathbb{E}F)$ of the following inequality 
$\|(G^{2}+ \tau^{2} F^{2})^{1/2}  \|_{p} \leq C \|F\|_{p},$
where  $G$ is the transform of a martingale $F$ under a predictable sequence $\eps$ with absolute value 1, $1<p< 2$, and $\tau$ is any real number.
\end{abstract}

\date{}
\maketitle
\tableofcontents
\newpage

\setcounter{equation}{0}
\setcounter{Th}{0}


\section{Introduction} 

Let $I$ be an interval of the real line $\mathbb{R}$, and let  $|I|$  be its Lebesgue length. By symbol $\mathcal{B}$ we denote the  $\sigma$-algebra of Borel subsets of $I$.  
Let $\{ F_{n}\}_{n=0}^{\infty}$ be a martingale on the probability space $(I, \mathcal{B},dx/|I|)$ with a filtration $\{I,\emptyset\} = \mathcal{F}_{0} \subset \mathcal{F}_{1} \subset ... \subset \mathcal{F}$. Consider any sequence of functions $\{\eps_{n}\}_{n=1}^{\infty}$ such that for each $n \geq 1$, $\eps_{n}$ is $\mathcal{F}_{n-1}$ measurable and $|\eps_{n}| \leq 1$. Let $G_{0}$ be a constant function on $I$;  for any $n \geq 1$, let $G_n$ denote
\begin{align*}
G_{0} + \sum_{k=1}^{n}\eps_{k}(F_{k}-F_{k-1}).
\end{align*} 
The sequence $\{G_{n}\}_{n=0}^{\infty}$  is called the \emph{martingale transform} of $\{F_{n}\}$. Obviously $\{G_{n}\}_{n=0}^{\infty}$ is a martingale with the same filtration $\{ \mathcal{F}_{n}\}_{n=0}^{\infty}$. Note that since $\{F_{n}\}$ and $\{G_{n}\}$ are martingales, we have $F_{0}=\mathbb{E}F_{n}$ and $G_{0}=\mathbb{E}G_{n}$ for any $n\geq 0$. 

In~\cite{Bu1} Burkholder proved that if  $|G_{0}|\leq |F_{0}|,$  $1<p<\infty$,  then we have the sharp estimate
\begin{align}\label{bure}
\left \| G_{n}\right\| _{L^{p}}\leq (p^{*}-1)\|F_{n}\|_{L^{p}}\quad \text{for all} \quad  n \geq 0,
\end{align}
where $p^{*}-1 = \max\{p-1, \frac{1}{p-1} \}.$ Burkholder showed that it is sufficient to prove  inequality (\ref{bure})  for the sequences of numbers $\{\eps_{n}\}$ such that $\eps_{n} = \pm 1$ for all   $n\geq 1$. It was also mentioned that such an estimate as (\ref{bure}) does not depend on the choice of filtration  $\{\mathcal{F}_{n}\}$. For example, one can consider only the dyadic filtration. For more information on the estimate (\ref{bure}) we refer the reader to~\cite{Bu1}, ~\cite{Choi}.

In~\cite{VaVo1} the result was slightly generalized by Bellman function technique and Monge--Amp\`ere equation, i.e., the estimate (\ref{bure}) holds if and only if 
\begin{align}\label{convavo}
|G_{0}| \leq (p^{*}-1)|F_{0}|.
\end{align}

In what follows we assume that $\{ \eps_{n}\}$ is a predictable sequence of functions such that $|\eps_{n}| =1$. 

In~\cite{BJV}, a perturbation of the martingale transform was investigated. Namely, under the same assumptions as (\ref{convavo}) it was proved that for $2\leq p<\infty$, $\tau \in \mathbb{R},$ we have the sharp estimate
\begin{align}\label{bor}
\|(G_{n}^{2}+\tau^{2}F_{n}^{2})^{1/2}\|_{L^{p}} \leq ((p^{*}-1)^{2}+\tau^{2})^{1/2} \|F_{n}\|_{L^{p}}, \quad \text{for all} \quad  n \geq 0.
\end{align}
 It was also claimed to be proven  that the same sharp estimate holds for $1<p<2$, $|\tau| \leq 0.5$, and the case $1<p<2, \;|\tau|>0.5$ was left open.
 
 The inequality  (\ref{bor}) stems from important questions concerning the $L^{p}$ bounds for the perturbation of Beurling--Ahlfors operator and hence it is of interest. We refer the reader to recent works regarding martingale inequalities and estimates of Beurling--Ahlfors operator \cite{RBPJ, RBPJM, RBAO, RBGW, BJV} and references therein.

We should mention that Burkholder's method~\cite{Bu1} and the Bellman function approach~\cite{VaVo1},~\cite{BJV} have similar traces  in the sense that both of them reduce the required estimate to finding  a certain minimal diagonally concave function with prescribed boundary conditions. However, the methods of construction of such a function are different. Unlike Burkholder's method~\cite{Bu1}, in~\cite{VaVo1} and ~\cite{BJV} the construction of the function is based on the Monge--Amp\`ere equation.

\subsection{Our main results}

Firstly, we should mention that the proof of (\ref{bor}) presented in~\cite{BJV} has a gap in the case $1<p<2$, $0<|\tau|\leq 0.5$ (the constructed function does not satisfy necessary concavity condition).

 In the present paper  we obtain the sharp $L^p$ estimate of the perturbed martingale transform for the  remaining case $1<p<2$ and for all $\tau \in \mathbb{R}$. Moreover, we do not require condition (\ref{convavo}).

We define
\begin{align*}
u(z) \df 
\tau^{p}(p-1)\left(\tau^{2}+z^{2}\right)^{(2-p)/2}-\tau^{2}(p-1)+(1+z)^{2-p}-z(2-p)-1.
\end{align*}

\begin{Th}\label{fm1}

Let $1 <p<2,$ and let  $\{G_{n}\}_{n=0}^{\infty}$ be a martingale transform of $\{F_{n}\}_{n=0}^{\infty}.$ Set $\beta = \frac{|G_{0}|-|F_{0}|}{|G_{0}|+|F_{0}|}$. The following estimates are sharp:
\begin{itemize}
\item[1.]
If $u\left(\frac{1}{p-1}\right) \leq 0$ then
\begin{align*}
\| (\tau^{2}F_{n}^{2}+G_{n}^{2})^{1/2}\|_{L^{p}} \leq \left(\tau^{2}+\max\left\{\left|\frac{G_{0}}{F_{0}}\right|,\frac{1}{p-1}\right\}^{2} \right)^{\frac{1}{2}}\|F_{n}\|_{L^{p}}, \quad \text{for all} \quad  n \geq 0.
\end{align*} 
\item[2.]
If $u\left(\frac{1}{p-1}\right) > 0$ then
\begin{align*}
\|(\tau^{2}F_{n}^{2}+G_{n}^{2})^{1/2}\|_{L^{p}}^{p} \leq C(\beta)\|F_{n}\|_{L^{p}}^{p}, \quad \text{for all} \quad  n \geq 0,
\end{align*}
where  $C(\beta)$ is continuous  nondecreasing, and it is  defined as follows: 
\begin{align*}
C(\beta)\df 
\displaystyle
\begin{cases} 
\left(\tau^{2}+\frac{|G_{0}|^{2}}{|F_{0}|^{2}} \right)^{p/2}, & \beta \geq s_{0};\\
\displaystyle
\frac{\tau^{p}}{1-\frac{2^{2-p}(1-s_{0})^{p-1}}{(\tau^{2}+1)(p-1)(1-s_{0})+2(2-p)}}, & \beta \leq  -1+\frac{2}{p};\\
\displaystyle
C(\beta), & \beta \in (-1+2/p, s_{0});\\
\end{cases}
\end{align*}
where $s_{0} \in (-1+2/p,1)$ is the solution of the equation $u\left( \frac{1+s_{0}}{1-s_{0}}\right)=0.$
\end{itemize}
\end{Th}
Explicit expression for the function $C(\beta)$ on the interval $(-1+2/p,s_{0})$ was hard to present in a simple way.  The reader can find the value of the function $C(\beta)$  in Theorem~\ref{fullth}, part $\textup{(ii)}$. 

\begin{Rem}
The condition $u\left(\frac{1}{p-1}\right) \leq 0$ holds when $|\tau| \leq 0.822$. So we also obtain Burkholder's result in the limit case when $\tau = 0$. It is worth mentioning that although the proof of the estimate (\ref{bor})  has a gap in \cite{BJV}, the claimed result in the case $1<p<2$, $|\tau|<0.5$ remains true as a result of Theorem \ref{fm1}.
\end{Rem}

One of the important results  of the current paper is that we  find the function (\ref{bellmanf}), and the above estimates are corollaries of  this result. We would like to mention that unlike~\cite{VaVo1} and~\cite{BJV} the argument exploited in the  current paper is different. Instead of writing a lot of technical computations and checking which case is valid, we present some pure geometrical facts regarding minimal concave functions with prescribed boundary conditions, and by this way we avoid  computations. Moreover, we explain to the reader  how  we construct our Bellman function (\ref{bellmanf}) based on these geometrical facts derived in Section~\ref{geom}.

\subsection{Plan of the paper}
In Section~\ref{sec2} we formulate results about how to reduce the estimate (\ref{bor}) to finding of a certain function with required properties. 
 These results are  well-known and can be found in ~\cite{BJV}. A slightly different function was investigated in ~\cite{VaVo1}, however, it possesses almost the same properties and the proof works exactly in the same way.  We only mention these results and the fact that we look for a minimal continuous  diagonally concave function $H(x_{1},x_{2},x_{3})$ (see Definition~\ref{biconcavefunction}) in the domain $\Omega = \{(x_{1},x_{2},x_{3}) \in \mathbb{R}^{3} : |x_{1}|^{p} \leq x_{3} \}$ with the boundary condition $H(x_{1},x_{2},|x_{1}|^{p}) = (x_{2}^{2}+\tau^{2}x_{1}^{2})^{p/2}$.

 Section \ref{geom} is devoted to the investigation of the minimal concave functions in two variables. It is worth mentioning that the first crucial steps in this direction for some special cases were made in~\cite{iOSvz} (see also \cite{iOSvz1, iOSvz2}). In Section \ref{geom} we develop this theory for a slightly more general case. We investigate some special foliation called the \emph{cup} and  another useful object, called  \emph{force functions}. 
 
 We should note that the theory of minimal concave functions in two variables does not include the minimal diagonally concave functions in three variables. Nevertheless, this  knowledge allows us to construct the candidate for $H$ in Section \ref{construction}, but with some additional technical work not mentioned in Section \ref{geom}. 
 
In section  \ref{sharp} we find the good estimates for the perturbed martingale transform. In Section \ref{fol}  we prove that the candidate for $H$ constructed in Section~\ref{construction} coincides with $H$, and as a corollary  we show the sharpness of the estimates found for the perturbed martingale transform in Section~\ref{sharp}.
 
In conclusion,  the reader can note that the hard technical part of the current paper lies in the construction of the minimal diagonally concave function in three variables with the given boundary condition. 

\section{Definitions and known results}\label{sec2}



Let  $\mathbb{E} F \df \av{F}{I}$
where 
\begin{align*}
\av{F}{J} \df \frac{1}{|J|}\int_{J}F(t)dt
\end{align*}

for any interval $J$ of the real line. Let $F$ and $G$ be real valued integrable functions. Let $G_{n}=\mathbb{E}(G|\mathcal{M}_{n})$ and $F_{n}=\mathbb{E}(F|\mathcal{M}_{n})$ for $n\geq 0$, where \{$\mathcal{M}_{n}$\} is a dyadic filtration (see \cite{BJV}).


 \begin{Def}
 If the martingale $\{G_{n}\}$  satisfies $|G_{n+1}-G_{n}|=|F_{n+1}-F_{n}|$ for each $n \geq 0$, then $G$ is called the martingale transform of  $F$.
\end{Def}
Recall that we are interested in the  estimate 
\begin{align}\label{espirveli}
\|(G^{2}+\tau^{2}F^{2})^{1/2}\|_{L^{p}} \leq C \|F\|_{L^{p}}.
\end{align}
We introduce the Bellman function

\begin{align}
H(\s{x}) \df  \sup_{F,G}\{ \mathbb{E}B(\varphi(F,G)),\; \mathbb{E}\varphi(F,G)=\s{x}, \; |G_{n+1}-G_{n}| = |F_{n+1}-F_{n}|,  n\geq 0 \}. \label{bellmanf}
\end{align} 
where $\varphi(x_{1},x_{2}) = (x_{1},x_{2},|x_{1}|^{p})$, $B(\varphi(x_{1},x_{2})) = (x_{2}^{2}+\tau^{2}x_{1}^{2})^{p/2}$, $\s{x} = (x_{1},x_{2},x_{3})$. 
\begin{Rem}
In what follows bold lowercase letters denote points in $\mathbb{R}^{3}$.
\end{Rem}
Then we see that the estimate (\ref{espirveli}) can be rewritten as follows:
\begin{align*}
H(x_{1},x_{2},x_{3}) \leq C^{p} x_{3}.
\end{align*}

We mention that the Bellman function $H$ does not depend on the choice of the interval $I$. Without loss of generality we may assume that $I=[0,1]$.


\begin{Def}\label{admissible}
Given a point  $\s{x}\in \mathbb{R}^{3}$, a pair $(F,G)$ is said to be admissible for $\s{x}$ if $G$ is the martingale transform of $F$ and $\mathbb{E} (F,G,|F|^{p})=\s{x}$.
\end{Def}
\begin{Prop}
The domain of  $H(\s{x})$ is  $\Omega = \{(x_{1},x_{2},x_{3}) \in \mathbb{R}^{3} : |x_{1}|^{p} \leq x_{3} \}$,  and  $H$ satisfies the boundary condition 
\begin{align}\label{boundarycondition}
H(x_{1},x_{2},|x_{1}|^{p}) =  (x_{2}^{2}+\tau^{2}x_{1}^{2})^{p/2}.
\end{align}
\end{Prop}

\begin{Def}\label{biconcavefunction}
A function $U$ is said to be  diagonally concave in $\Omega$, if  it is concave in both $\Omega \cap \{(x_{1},x_{2},x_{3}) : \, x_{1}+x_{2}=A \}$ and $\Omega \cap \{(x_{1},x_{2},x_{3}) : \, x_{1}-x_{2}=A \}$ for every constant $A \in \mathbb{R}$.
\end{Def}
\begin{Prop}\label{propbic}
$H(\s{x})$ is a diagonally concave function in $\Omega$. 
\end{Prop}

\begin{Prop}\label{mazhorantochka}
If $U$ is a continuous diagonally concave function in $\Omega$ with boundary condition $U(x_{1},x_{2},|x_{1}|^{p}) \geq (x_{2}^{2}+\tau^{2}x_{1}^{2})^{p/2},$ then $U \geq H$ in $\Omega$. 
\end{Prop}


We explain our strategy of finding the Bellman function $H$.
We are going to find a minimal candidate $B$, that is  continuous, diagonally concave, with the fixed boundary condition $B|_{\partial \Omega} = (y^{2}+\tau^{2}x^{2})^{p/2}$. We warn the reader that the symbol $B$ denoted boundary data previously, however, in Section~\ref{fol} we are going to use symbol $B$ as the candidate for the minimal diagonally concave function. Obviously $ B\geq H$ by Proposition~\ref{mazhorantochka}.  We will also see that given $\s{x} \in \Omega$ and any $\eps >0$, we can construct an admissible pair  $(F,G)$  such that
$B(\s{x})<\mathbb{E}(F^{2}+\tau^{2}G^{2})^{p/2}+\eps$. This will show that $B \leq H$ and hence $B = H$.

In order to construct the minimal candidate $B$, we have to elaborate few preliminary concepts from differential geometry. We introduce notion of \emph{foliation}  and  \emph{force} functions.
\section{Homogeneous Monge--Amp\`ere equation and minimal concave functions} \label{geom}

\subsection{Foliation}
Let   $g(s) \in C^{3}(I)$  be such that $g''>0$, and let  $\Omega$ be a convex  domain which is  bounded by the curve $(s,g(s))$ and the tangents that pass through the end-points of the curve (see Figure~\ref{fig:basic}).
 Fix some function $f(s) \in C^{3}(I)$. The first question we ask is the following: how the minimal concave function $\Bell(x_{1},x_{2})$ with boundary data $\Bell(s,g(s)))=f(s)$ looks \emph{locally} in a subdomain of $\Omega$. In other words,  take a convex hull of the curve  $(s,g(s),f(s)), s \in I,$ then the question is how the boundary of this convex hull looks like. 

We recall that the concavity is equivalent to the following inequalities: 
\begin{align}
\det (\mathrm{d}^{2} \Bell) &\geq 0, \label{gg}\\
  \Bell''_{x_{1}x_{1}}+\Bell''_{x_{2}x_{2}} &\leq 0.\label{gm}
\end{align}
The expression (\ref{gg}) is the Gaussian curvature of the surface $(x_{1},x_{2},\Bell(x_{1},x_{2}))$ up to a positive factor $(1+(\Bell'_{x_{1}})^{2}+(\Bell'_{x_{2}})^{2})^{2}$. So in order to minimize the function $\s{B}(x_{1},x_{2})$, it is reasonable to minimize the Gaussian curvature. Therefore, we will 
look for a  surface with zero Gaussian curvature. Here the homogeneous Monge--Amp\`ere equation arises.  These surfaces are known as \emph{developable surfaces} i.e., such a surface can be constructed by bending a plane region. The important property of such surfaces is that they consist of line segments, i.e., the function $\Bell$ satisfying homogeneous Monge--Amp\`ere equation $\det (\mathrm{d}^{2} \Bell)=0$ is linear along some {\em family of segments}. These considerations lead us to investigate such functions $\Bell$.  Firstly, we define a {\em foliation}.  For any segment $\ell$ in the Euclidean space by symbol $\ell^{\circ}$ we denote an open segment i.e., $\ell$ without endpoints. 

\begin{wrapfigure}[16]{r}{0pt}
\includegraphics[scale=0.8]{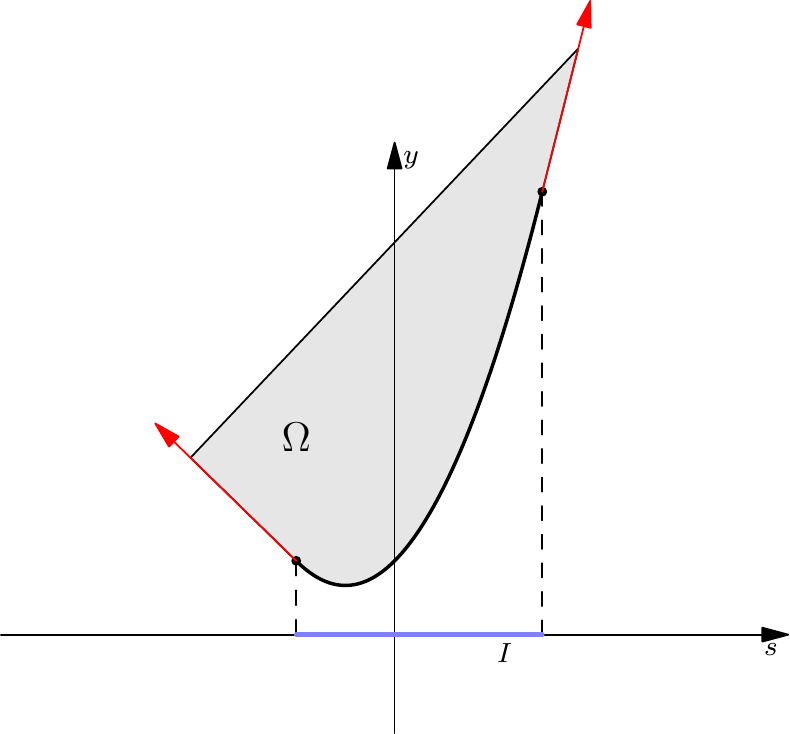}
\caption{Domain $\Omega$}
\label{fig:basic}
\end{wrapfigure}
Fix any subinterval $J \subseteq I$. By symbol $\Theta(J,g)$ we denote an arbitrary set of nontrivial segments (i.e. single points are excluded)  in $\mathbb{R}^{2}$  with the following requirements:
\begin{itemize}
\item[1.] For any  $\ell \in \Theta(J,g)$ we have  $\ell^{\circ} \in \Omega$.
\item[2.] For any $\ell_{1}, \ell_{2} \in \Theta(J,g)$ we have $\ell_{1} \cap \ell_{2} = \emptyset$.  
\item[3.] For any $\ell \in \Theta(J,g)$ there exists only one point  $s\in J$  such that $(s,g(s))$ is one of the end-points of  the segment $\ell$ and vice versa, for any point $s\in J$ there exists $\ell \in \Theta(J,g)$ such that $(s,g(s))$ is one of the end-points of the segment $\ell$.
\item[4.] There exists $C^{1}$ smooth argument function $\theta(s)$.
\end{itemize}

We explain the meaning of the requirement 4. To each point $s \in J$ there  corresponds only one segment $\ell \in \Theta(J,g)$ with an endpoint $(s,g(s))$. Take a nonzero vector with initial point $(s,g(s))$, parallel to the segment $\ell$ and having an endpoint in $\Omega$. We define the value of  $\theta(s)$ to be an argument of this vector. Surely argument is defined up to additive number $2\pi k$ where $k \in \mathbb{Z}$. Nevertheless, we take any representative from these angles. We do the same for all other points $s \in I$. In this way we get a family of functions $\theta(s)$. If there exists $C^{1}(J)$ smooth function $\theta(s)$ from this family then the requirement 4 is satisfied.  
\begin{Rem}
It is clear that if $\theta(s)$ is $C^{1}(J)$ smooth argument function, then for any  $k \in \mathbb{Z}$, $\theta(s)+2\pi k$ is also $C^{1}(J)$ smooth argument function. Any two $C^{1}(J)$ smooth argument functions differ by constant $2\pi n$ for some $n \in \mathbb{Z}$. 
\end{Rem}
This remark is the consequence of the fact that the quantity $\theta'(s)$ is well defined regardless of the choices of $\theta(s)$. 
Next, we define $\Omega(\Theta(J,g)) = \cup_{\ell \in \Theta(J,g)}\ell^{\circ}$. Given a point $x \in \Omega(\Theta(J,g))$ we denote by $\ell(x)$ a segment  $\ell(x) \in \Theta(J,g)$ which passes through the point $x$. If $x = (s,g(s))$ then instead of $\ell((s,g(s)))$ we just write $\ell(s)$.  Surely such a segment exists, and it is unique.  We denote by $s(x)$  a point $s(x) \in J$ such that $(s(x),g(s(x)))$ is  one of the end points of the segment $\ell(x)$. Moreover, in a natural way we set $s(x)=s$ if $x =(s,g(s))$. It is clear that such $s(x)$ exists, and it is unique. We introduce a function 
\begin{align}\label{kfunction}
K(s) = g'(s) \cos \theta(s) - \sin \theta(s), \qquad s \in J.
\end{align}

 Note that that $K <0$. This inequality becomes obvious if we rewrite  $g'(s) \cos \theta(s) - \sin \theta(s) = \langle (1,g'), (-\sin \theta, \cos \theta)\rangle $  and take into account the requirement 1 of $\Theta(J,g)$. Note that $\langle \cdot, \cdot \rangle$ means scalar product in Euclidean space. 
 \begin{wrapfigure}[13]{r}{0pt}
\includegraphics[scale=0.8]{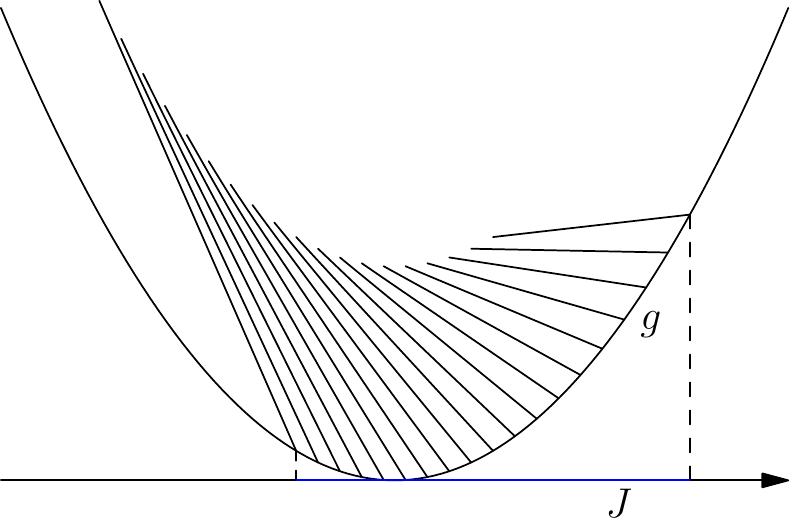}
\caption{Foliation $\Theta(J,g)$}
\label{fig:foli}
\end{wrapfigure}
   We need few more  requirements on $\Theta(J,g)$.
\begin{itemize}
\item[5.]  For any  $x=(x_{1},x_{2}) \in \Omega(\Theta(J,g))$ we have an inequality 
$K(s(x))+\theta'(s(x)) \|(x_{1}-s(x),x_{2}-g(s(x))) \|<0.$
\item[6.]  The function $s(x)$ is continuous  in $\Omega(\Theta(J,g))\cup \Gamma(J)$ where $\Gamma(J)= \{ (s,g(s)) : \; s\in J\}$.
\end{itemize}

Note that if $\theta'(s) \leq 0$ (which happens in most of the cases) then the requirement 5 holds. If we know the endpoints of the segments $\Theta(J,g)$, then in order to verify the requirement 5 it is enough to check at those points $x=(x_{1},x_{2}),$ where $x$ is the another endpoint of the segment other than $(s,g(s))$. Roughly speaking the requirement 5 means the segments of $\Theta(J,g)$ do not rotate rapidly counterclockwise.

\begin{Def}
A set of segments $\Theta(J, g)$ with the requirements mentioned above is called {\em foliation}. The set $\Omega(\Theta(J,g))$ is called {\em domain of foliation}. 
\end{Def}
A typical example of a foliation is given in Figure~\ref{fig:foli}.
\begin{Le}\label{chervyak}
The function $s(x)$ belongs to $C^{1}(\Omega(\Theta(J,g)))$. Moreover
\begin{align}\label{shift}
(s'_{x_{1}},s'_{x_{2}})  = \frac{(\sin \theta, -\cos \theta)}{-K(s)-\theta' \cdot  \|(x_{1}-s,x_{2}-g(s)) \| }.
\end{align} 
\end{Le}

\begin{proof}
Definition of the function $s(x)$ implies that 
\begin{align*}
-(x_{1}-s)\sin \theta(s)+(x_{2}-g(s))\cos \theta(s)=0.
\end{align*}
Therefore the lemma is an immediate consequence of the implicit function theorem. 
\end{proof}

Let $J=[s_{1},s_{2}] \subseteq I,$ and  let $(s,g(s),f(s))\in C^{3}(I)$ be  such that $g'' > 0$ on $I$. Consider an arbitrary foliation $\Theta(J,g)$ with an arbitrary $C^{1}([s_{1},s_{2}])$ smooth argument function $\theta(s)$. We need the following technical lemma.
\begin{Le}\label{texnika}
The solutions of the system of equations 
\begin{align} 
& t'_{1}(s)\cos \theta(s)+t'_{2}(s)\sin \theta(s)=0, \label{extremal}\\
& t_{1}(s)+t_{2}(s)g'(s) =f'(s),\quad s \in J \label{bdcondition}
\end{align}
 are the following functions
\begin{align*}
&t_{1}(s) =  \int_{s_{1}}^{s}\left( \frac{g''(r)}{K(r)}\sin \theta(r) \cdot t_{2}(r) -\frac{f''(r)}{K(r)}\sin \theta(r)\right) dr + f'(s_{1})-t_{2}(s_{1})g'(s_{1}),\\
&t_{2}(s) = t_{2}(s_{1}) \exp\left(-\int_{s_{1}}^{s}\frac{g''(r)}{K(r)}\cos \theta(r) dr\right)+ \int_{s_{1}}^{s} \frac{f''(y)}{K(y)}\exp\left(-\int_{y}^{s}\frac{g''(r)}{K(r)}\cos \theta(r) dr\right)  \cos \theta(y) dy, \quad s \in J
\end{align*}
where $t_{2}(s_{1})$ is an arbitrary real number. 
\end{Le}
\begin{proof}
We differentiate (\ref{bdcondition}) and combine it with (\ref{extremal}) to obtain the system 
\begin{align*}
  \begin{pmatrix}
  \cos \theta & \sin \theta \\
  1 & g'
 \end{pmatrix}
 \begin{pmatrix}
  t_{1}'\\
  t_{2}' 
 \end{pmatrix}
 =
   \begin{pmatrix}
  0 & 0 \\
  0 & -g''
 \end{pmatrix}
  \begin{pmatrix}
  t_{1}\\
  t_{2} 
 \end{pmatrix}
 +
   \begin{pmatrix}
  0\\
  f''
 \end{pmatrix}.
\end{align*}
This implies that 
\begin{align}\label{gdiffur}
 \begin{pmatrix}
  t_{1}'\\
  t_{2}' 
 \end{pmatrix}
 =
 \frac{g''}{K}
   \begin{pmatrix}
  0 &  \sin \theta \\
  0 & - \cos \theta
 \end{pmatrix}
  \begin{pmatrix}
  t_{1}\\
  t_{2} 
 \end{pmatrix}
 +
  \frac{f''}{K}
   \begin{pmatrix}
  -\sin \theta\\
  \cos \theta
 \end{pmatrix}.
\end{align}
By solving this system of differential equations and using the fact that $t_{1}(s_{1})+g'(s_{1})t_{2}(s_{1})=f'(s_{1})$ we get the desired result. 
\end{proof}
\begin{Rem}\label{zoloto}
Integration by parts allows us to rewrite the expression for $t_{2}(s)$ as follows 
\begin{align*}
&t_{2}(s) = \exp\left(-\int_{s_{1}}^{s}\frac{g''(r)}{K(r)}\cos \theta(r) dr\right) \left(t_{2}(s_{1})-\frac{f''(s_{1})}{g''(s_{1})} \right)+\frac{f''(s)}{g''(s)}-\\
&-\int_{s_{1}}^{s} \left[\frac{f''(y)}{g''(y)}\right]'   \exp\left(-\int_{y}^{s}\frac{g''(r)}{K(r)}\cos \theta(r) dr\right) dy.
\end{align*}
\end{Rem}

\begin{Def}
We say that a function $\Bell$ has a foliation $\Theta(J, g)$  if it is continuous on $\Omega(\Theta(J,g))$, and it is linear on each segment of $\Theta(J,g)$. 
\end{Def}

The following lemma describes how to construct a function $\Bell$ with a given foliation $\Theta(J, g)$, boundary condition $\Bell(s,g(s))=f(s)$, such that $\Bell$ satisfies the homogeneous Monge--Amp\`ere equation.

Consider a function $\Bell$ defined as follows 
\begin{align}
\Bell(x) = f(s) + \langle t(s), x-(s,g(s))\rangle, \quad x = (x_{1},x_{2}) \in \Omega(\Theta(J,g))  \label{bellmanfunction}
\end{align}
where $s=s(x)$, and $t(s) = (t_{1}(s),t_{2}(s))$ satisfies the system of the equations (\ref{extremal}), (\ref{bdcondition})  with an arbitrary $t_{2}(s_{1})$.
\begin{Le}\label{dev}
The function $\Bell$  defined by (\ref{bellmanfunction}) satisfies the following properties:
\begin{itemize}
\item[1.]$\Bell \in C^{2}(\Omega(\Theta(J,g)))\cap C^{1}(\Omega(\Theta(J,g))\cup \Gamma)$, $\Bell$ has the foliation $\Theta(J,g)$  and 
\begin{align}\label{bcondition0}
\Bell(s,g(s))=f(s) \quad \text{for all} \quad  s \in [s_{1},s_{2}].
\end{align}
\item[2.] $\nabla \Bell(x) = t(s)$, where $s=s(x)$, moreover  $\Bell$ satisfies the homogeneous Monge--Amp\`ere equation.
\end{itemize}
\end{Le}  
\begin{proof}
The fact that $\Bell$ has the foliation $\Theta(J,g)$, and it satisfies the equality~(\ref{bcondition0}) immediately follows from the definition of the function 
$\Bell$. We check the condition of smoothness.  By Lemma~\ref{chervyak} and Lemma~\ref{texnika}  we have $s(x) \in C^{2}(\Omega(\Theta(J,g)))$  and $t_{1},t_{2} \in C^{1}(J)$, therefore the right-hand side of (\ref{bellmanfunction}) is differentiable with respect to $x$. So after differentiation of (\ref{bellmanfunction}) we get 
\begin{align}
\nabla \Bell(x) =  \left[ f'(s) - \langle t(s), (1,g'(s))\rangle \right] (s'_{x_{1}},s'_{x_{2}}) + t(s) + \langle t'(s),x-(s,g(s))\rangle (s'_{x_{1}},s'_{x_{2}}). \label{bescar}
\end{align}
Using (\ref{extremal}) and (\ref{bdcondition}) we obtain $\nabla \Bell (x) = t(s)$. Taking derivative with respect to $x$ the second time we get 
\begin{align*}
\frac{\partial^{2} \Bell}{\partial x_{1}^{2}}= t'_{1}(s) s'_{x_{1}}, \quad \frac{\partial^{2} \Bell}{\partial x_{2} \partial x_{1}}= t'_{1}(s) s'_{x_{2}}, \quad \frac{\partial^{2} \Bell}{\partial x_{1} \partial x_{2} }= t'_{2}(s) s'_{x_{1}}, \quad  \frac{\partial^{2} \Bell}{\partial x_{2}^{2}}= t'_{2}(s) s'_{x_{2}}.
\end{align*}
Using (\ref{extremal}) we get that $t'_{1}(s)s'_{x_{2}} = t'_{2}(s) s'_{x_{1}}$, therefore $\Bell \in C^{2}(\Omega(\Theta(J,g)))$. Finally, we check that 
$\Bell$ satisfies the homogeneous Monge--Amp\`ere equation. Indeed, 
\begin{align*}
&\det(\text{d}^{2}\Bell)= \frac{\partial^{2} \Bell}{\partial x_{1}^{2}} \cdot \frac{\partial^{2} \Bell}{\partial x_{2}^{2}} - \frac{\partial^{2} \Bell}{\partial x_{2} \partial x_{1}} \cdot \frac{\partial^{2} \Bell}{\partial x_{1} \partial x_{2} } =t'_{1}(s) s'_{x_{1}} \cdot t'_{2}(s) s'_{x_{1}} - t'_{1}(s) s'_{x_{2}} \cdot t'_{2}(s) s'_{x_{1}} = 0.
\end{align*}
  \end{proof}
  \begin{Def}
  The function $t(s)=(t_{1}(s),t_{2}(s))=\nabla \Bell(x)$, $s=s(x)$, is called {\em gradient function} corresponding to $\Bell$. 
  \end{Def}
  The following lemma investigates the concavity of the function $\Bell$ defined by (\ref{bellmanfunction}). Let $\|\tilde \ell (x)\| = \|(s(x)-x_{1}, g(s(x))-x_{2})\|$, where $x = (x_{1},x_{2}) \in \Omega(\Theta(J,g))$.
\begin{Le}\label{traceeq}
The following equalities hold 
\begin{align*}
&\frac{\partial^{2} \Bell}{\partial x_{1}^{2}}+\frac{\partial^{2} \Bell}{\partial x_{2}^{2}}=\frac{g''}{K(K+\theta' \| \tilde\ell (x)\|)}\left(-t_{2} +\frac{f''}{g''}\right)= \\
&\frac{g''}{K(K+\theta' \| \tilde \ell (x)\|)} \times \left[ -\exp\left(-\int_{s_{1}}^{s}\frac{g''(r)}{K(r)}\cos \theta(r) dr\right) \left(t_{2}(s_{1})-\frac{f''(s_{1})}{g''(s_{1})} \right)\right.\\
&\left.  +\int_{s_{1}}^{s} \left[\frac{f''(y)}{g''(y)}\right]'   \exp\left(-\int_{y}^{s}\frac{g''(r)}{K(r)}\cos \theta(r) dr\right) dy\right].
\end{align*}
\end{Le}
\begin{proof} Note that 
\begin{align*}
\frac{\partial^{2} \Bell}{\partial x_{1}^{2}}+\frac{\partial^{2} \Bell}{\partial x_{2}^{2}} =t'_{1}(s)s'_{1} + t'_{2}(s)s'_{2}.
\end{align*}
Therefore the lemma is a direct computation and  application of Equalities (\ref{shift}), (\ref{extremal}), (\ref{bdcondition}) and Remark~\ref{zoloto}.
\end{proof}
Finally, we get the following important statement. 
\begin{Cor}\label{sledstviesily}
The function $\Bell$ is concave in $\Omega(\Theta(J,g))$ if and only if $\mathcal{F}(s) \leq 0$, where  
\begin{align}
&\mathcal{F}(s) = -\exp\left(-\int_{s_{1}}^{s}\frac{g''(r)}{K(r)}\cos \theta(r) dr\right) \left(t_{2}(s_{1})-\frac{f''(s_{1})}{g''(s_{1})} \right) \label{silfun}\\
& +\int_{s_{1}}^{s} \left[\frac{f''(y)}{g''(y)}\right]'   \exp\left(-\int_{y}^{s}\frac{g''(r)}{K(r)}\cos \theta(r) dr\right) dy = \frac{f''(s)}{g''(s)}-t_{2}(s).\nonumber
\end{align}
\end{Cor}

\begin{proof}
$\Bell$ satisfies the homogeneous Monge--Amp\`ere equation. Therefore $\Bell$ is concave if and only if 
\begin{align}\label{trinequality}
\frac{\partial^{2} \Bell}{\partial x_{1}^{2}}+\frac{\partial^{2} \Bell}{\partial x_{2}^{2}}  \leq 0.
\end{align}
 Note that
\begin{align*}
\frac{g''}{K(K+\theta' \| \tilde \ell (x) \|)} >  0.
\end{align*}
Hence, according to Lemma~\ref{traceeq}, the inequality (\ref{trinequality}) holds if and only if $\mathcal{F}(s) \leq 0$. 
\end{proof}
Furthermore, the function $\mathcal{F}$ will be called {\em force} function. 
\begin{Rem}
The fact  $t_{2}(s) = f''/g'' - \mathcal{F} $  together with (\ref{gdiffur}) imply  that the force function $\mathcal{F}$ satisfies the following differential equation
\begin{align}
&\mathcal{F}' + \mathcal{F} \cdot \frac{\cos \theta}{K}g'' - \left[\frac{f''}{g''}\right]'=0, \quad s \in J \label{maindif}\\
&\mathcal{F}(s_{1}) =\frac{f''(s_{1})}{g''(s_{1})} - t_{2}(s_{1}). \nonumber
\end{align}
\end{Rem}

We remind the reader that for an arbitrary smooth curve $\gamma  = (s,g(s),f(s)),$ the torsion has the following expression 
\begin{align*}
\frac{\det (\gamma',\gamma'',\gamma''')}{\|\gamma'\times \gamma'' \|^{2}} = \frac{f'''g''-g'''f''}{\|\gamma'\times \gamma'' \|^{2}}=\frac{(g'')^{2}}{\|\gamma'\times \gamma'' \|^{2}} \cdot \left[ \frac{f''}{g''}\right]'.
\end{align*}

\begin{Cor}\label{torsion}
If $\mathcal{F}(s_{1})\leq 0$ and the torsion of a curve $(s,g(s),f(s)),$ $s \in J$ is negative, then the function $\Bell$ defined by (\ref{bellmanfunction}) is concave.  
\end{Cor}
\begin{proof}
The corollary  is an immediate consequence of (\ref{silfun}).
\end{proof}
Thus, we see that the torsion of the boundary data plays a crucial role in the concavity of a surface with zero Gaussian curvature.   More detailed investigations about how we choose the constant $t_{2}(s_{1})$  will be given in  Subsection~\ref{cupsection}. 

Let $\Theta(J,g)$ and $\tilde \Theta(J,g)$ be foliations with some argument functions $\theta(s)$ and $\tilde \theta (s)$ respectively. Let $\Bell$ and $\tilde \Bell$ be the corresponding functions defined by (\ref{bellmanfunction}), and let $\mathcal{F}, \tilde{\mathcal{F}}$ be the corresponding force functions. Note that  
 $\mathcal{F}(s) = \tilde{\mathcal{F}}(s)$  is equivalent to the equality $t(s)=\tilde t (s)$ where  $t(s) = (t_{1}(s),t_{2}(s))$ and $\tilde t (s)=(\tilde t_{1}(s),t_{2}(s))$ are the corresponding gradients of $\Bell$ and $\tilde \Bell$ (see (\ref{bdcondition}) and Corollary~\ref{sledstviesily}).
 
 \begin{wrapfigure}[13]{r}{0pt}
\includegraphics[scale=0.80]{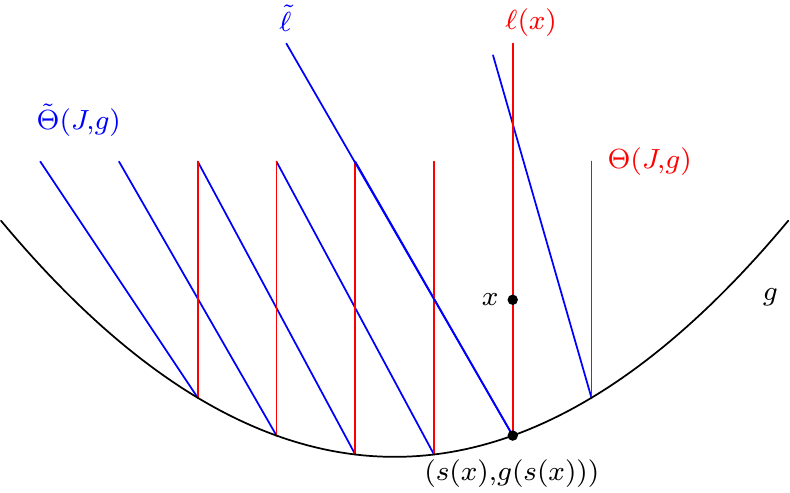}
\caption{Foliations $\Theta(J,g)$ and $\tilde{\Theta}(J,g)$}
\label{fig:rost}
\end{wrapfigure}
  Assume that the functions $\Bell$ and $\tilde \Bell$ are concave functions.
 \begin{Le}\label{comparison}
 If $\sin(\tilde \theta - \theta)\geq 0 $ for all  $ s \in J$,   and $\mathcal{F}(s_{1})=\tilde{\mathcal{F}}(s_{1})$, then $\tilde \Bell \leq \Bell $ on $\Omega(\Theta(J,g))\cap \tilde{\Omega}(\Theta(J,g)).$
 \end{Le}
In other words, the lemma says that if at initial point $(s_{1},g(s_{1}))$ gradients of the functions $\tilde \Bell$ and $\Bell$ coincide, and the foliation $\tilde{\Theta}(J,g)$ is ``to the left of''  the foliation $\Theta(J,g)$ (see Figure~\ref{fig:rost}) then  $\tilde{\Bell} \leq \Bell$ provided $\Bell$ and $\tilde \Bell$ are concave.  
\begin{proof}
Let $K$ and $\tilde K$ be the corresponding functions of $\Bell$ and $\tilde \Bell$ defined by (\ref{kfunction}). The condition  $K, \tilde K <0$ implies that the
inequality $\sin(\tilde \theta - \theta)\geq 0$ is equivalent to the inequality 
\begin{align}\label{pustyak}
\frac{\cos \tilde \theta}{\tilde K} \geq \frac{\cos \theta}{ K} \quad \text{for} \quad s\in J.
\end{align}
Indeed, if we rewrite (\ref{pustyak}) as $K \cos \tilde \theta \geq\tilde K \cos \theta$ then this simplifies to $-\sin\theta \cos \tilde \theta \geq -\sin \tilde{\theta} \cos \theta$, so the result follows. 
 \begin{wrapfigure}[13]{r}{0pt}
\includegraphics[scale=0.80]{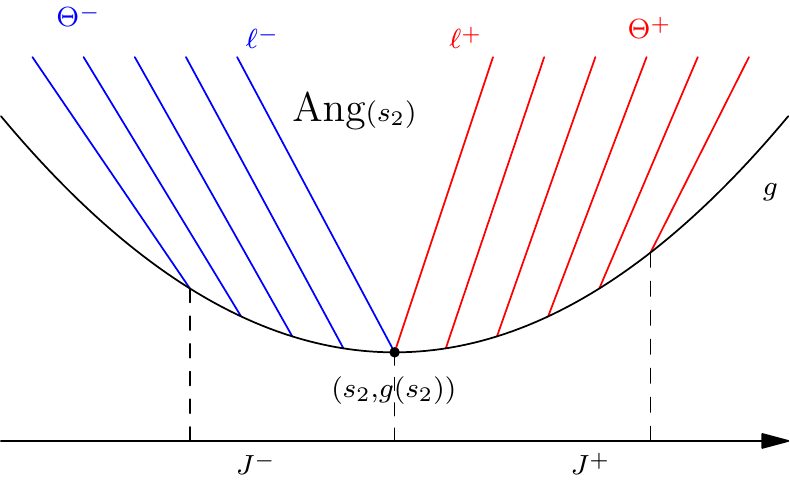}
\caption{Gluing of $\Bell^{-}$ and $\Bell^{+}$}
\label{fig:gluing}
\end{wrapfigure}
The force functions $\mathcal{F}, \tilde{\mathcal{F}}$ satisfy the differential equation (\ref{maindif}) with  the same boundary condition $\mathcal{F}(s_{1}) = \tilde{\mathcal{F}}(s_{1})$. Then by (\ref{pustyak}) and by  comparison theorems we get $\tilde{\mathcal{F}} \geq \mathcal{F}$ on $J$. This and (\ref{silfun}) imply that $\tilde t _{2} \leq t_{2}$ on $J$.  Pick any point $x \in \Omega(\Theta(J,g))\cap \tilde{\Omega}(\Theta(J,g)).$ Then there exists a segment $\ell(x) \in \Theta(J,g)$. Let $(s(x),g(s(x)))$ be the corresponding endpoint of this segment. There exists a segment $\tilde \ell \in \tilde{\Theta}(J,g)$ which has $(s(x),g(s(x)))$ as an endpoint (see Figure~\ref{fig:rost}).

 Consider a tangent plane $L(x)$ to $(x_{1},x_{2}, \tilde \Bell)$  at point $(s(x),g(s(x))).$ The fact that the gradient of  $\tilde \Bell$ is constant on $\tilde{\ell}$,  implies that   $L$ is tangent to $(x_{1},x_{2},\tilde \Bell )$ on $\tilde{\ell}$. Therefore
\begin{align*}
&L(x) = f(s)+ \langle (\tilde t _{1}(s),\tilde t _{2}(s)),(x_{1}-s,x_{2}-g(s)) \rangle,
\end{align*}
where $x=(x_{1},x_{2})$ and $s=s(x)$. 
  Concavity of  $\tilde{\Bell}$ implies that a value of the function $\tilde \Bell$ at point $y$ seen from the point $(s(x),g(s(x)))$ is less than $L(y)$. In particular  $\tilde \Bell(x) \leq L(x)$. Now it is enough to prove that $L(x) \leq \Bell(x)$. By (\ref{bellmanfunction}) we have
\begin{align*}
\Bell(x) = f(s)+ \langle (t _{1}(s), t_{2}(s)),(x_{1}-s(x),x_{2}-g(s)) \rangle.
\end{align*} 
Therefore using  (\ref{bdcondition}), $\langle (-g',1), (x_{1}-s, x_{2}-g(s))\rangle  \geq 0$ and the fact that  $\tilde t _{2} \leq t_{2}$ we get the desired result. 
\end{proof}

Let $J^{-} = [s_{1},s_{2}]$ and $J^{+} = [s_{2},s_{3}]$ where $J^{-},J^{+} \subset I$. Consider arbitrary foliations $\Theta^{-}=\Theta^{-}(J^{-},g)$ and $\Theta^{+}=\Theta^{+}(J^{+},g)$ such that 
$\Omega(\Theta^{-}) \cap \Omega(\Theta^{+}) = \emptyset$, and let $\theta^{-}$ and $\theta^{+}$ be the corresponding argument functions.
Let $\Bell^{-}$ and $\Bell^{+}$ be the corresponding functions defined by (\ref{bellmanfunction}), and let $t^{-}=(t_{1}^{-},t_{2}^{-})$, $t^{+}=(t_{1}^{+},t_{2}^{+})$ be the corresponding gradient functions. 
 Set $\operatorname{Ang}(s_{2})$ to be a convex hull of $\ell^{-}(s_{2})$ and $\ell^{+}(s_{2})$ where $\ell^{-}(s_{2})\in \Theta^{-}$, $\ell^{+}(s_{2})\in \Theta^{+}$ are the segments with the endpoint $(s_{2},g(s_{2}))$ (see Figure~\ref{fig:gluing}). We require that  $\operatorname{Ang}(s_{2})\cap \Omega(\Theta^{-})=\ell^{-}$ and $\operatorname{Ang}(s_{2})\cap \Omega(\Theta^{+})=\ell^{+}$.
 
  Let $\mathcal{F}^{-}, \mathcal{F}^{+}$ be the corresponding forces, and let $\Bell_{\operatorname{Ang}}$ be the function defined linearly on $\operatorname{Ang}(s_{2})$ via the values of $\Bell^{-}$ and $\Bell^{+}$ on $\ell^{-}$, $\ell^{+}$ respectively. 
\begin{Le}\label{gluing}
If $t_{2}^{-}(s_{2})=t_{2}^{+}(s_{2})$, then the function $\Bell$ defined as follows 
\begin{align*}\Bell(x) = 
\begin{cases}
\Bell^{-}(x), &   x \in \Omega(\Theta(J^{-},g)),\\
\Bell_{\operatorname{Ang}}(x), &   x \in \operatorname{Ang}(s_{2}),\\
\Bell^{+}(x), &  x \in \Omega(\Theta(J^{+},g)),\\
\end{cases}
\end{align*}
 belongs to the class $C^{1}(\Omega(\Theta^{-})\cup \operatorname{Ang}(s_{2}) \cup \Omega(\Theta^{+})\cup \Gamma(J^{-}\cup J^{+})).$
\end{Le} 
\begin{proof}
 By (\ref{bdcondition}) the condition $t_{2}^{-}(s_{2})=t_{2}^{+}(s_{2})$ is equivalent to the condition $t^{-}(s_{2})=t^{+}(s_{2})$. We recall that  the gradient of $\Bell^{-}$ is constant on $\ell^{-}(s_{2})$, and the gradient of $\Bell^{+}$ is constant on $\ell^{+}(s_{2})$, therefore the lemma follows immediately from the fact that $\Bell^{-}(s_{2},g(s_{2})) = \Bell^{+}(s_{2},g(s_{2})).$
\end{proof}
\begin{Rem}\label{razryv}
The fact $\Bell \in C^{1}$ implies that its gradient function $t(s) = \nabla \Bell$ is well defined, and it is continuous. Unfortunately, it is not necessarily true that $t(s) \in C^{1}([s_{1},s_{3}])$. However, it is clear that $t(s) \in C^{1}([s_{1},s_{2}])$, and $t(s) \in C^{1}([s_{2},s_{3}])$.
\end{Rem}

Finally we finish this section with the following important corollary about \emph{concave extension} of the functions with zero gaussian curvature.

Let $\Bell^{-}$ and $\Bell^{+}$ be defined as above (see Figure~\ref{fig:gluing}). Assume that $t_{2}^{-}(s_{2})=t_{2}^{+}(s_{2})$.

\begin{Cor}\label{rassh}
If $\Bell^{-}$ is concave in $\Omega(\Theta^{-})$ and the torsion of the curve $(s,g(s),f(s))$ is nonnegative on $J^{+}=[s_{2},s_{3}]$ then the function $\Bell$  defined in Lemma~\ref{gluing} is concave in the domain $\Omega(\Theta^{-})\cup \operatorname{Ang}(s_{2}) \cup \Omega(\Theta^{+})$.
\end{Cor}
In other words the corollary tells us that if we have constructed concave function $\Bell^{-}$ which satisfies homogeneous Monge--Amp\`ere equation, and we glued $\Bell^{-}$ smoothly with $\Bell^{+}$ (which also satisfies homogeneous Monge--Amp\`ere equation), then the result $\Bell$ is concave function provided that the space curve $(s,g(s),f(s))$ has nonnegative torsion on the interval $J^{+}$. 
\begin{proof}
By Lemma~\ref{sledstviesily} concavity of $\Bell^{-}$ implies $\mathcal{F}^{-}(s_{2})\leq 0$. By (\ref{silfun})
the condition $t_{2}^{-}(s_{2})=t_{2}^{+}(s_{2})$ is equivalent to $\mathcal{F}^{-}(s_{2})=\mathcal{F}^{+}(s_{2})$. By Corollary~\ref{torsion} we get that $\Bell^{+}$ is concave. Thus, concavity of $\Bell$ follows from Lemma~\ref{gluing}.
\end{proof}

\subsection{Cup}\label{cupsection}

 \begin{wrapfigure}[12]{r}{0pt}
\includegraphics[scale=0.8]{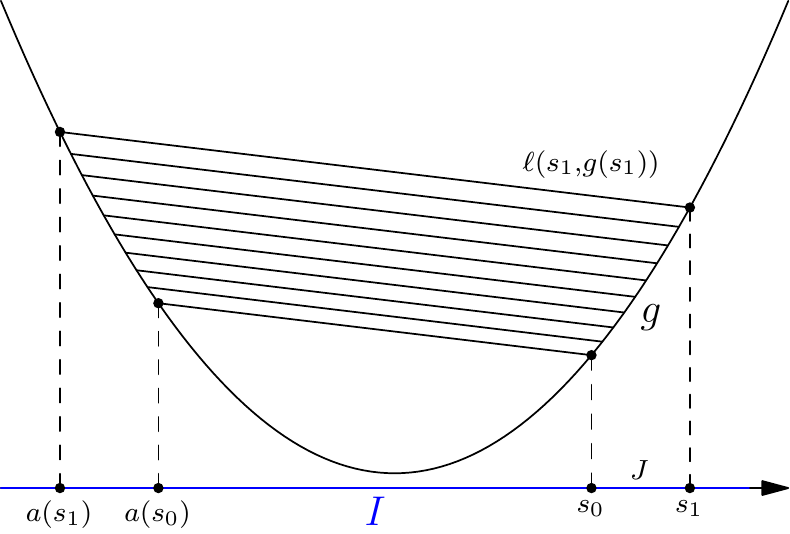}
\caption{Foliation $\Theta_{\operatorname{cup}}(J,g)$}
\label{fig:cup}
\end{wrapfigure}
In this subsection we are going to consider a special type of foliation which is called {\em Cup}.
Fix an interval $I$ and consider an arbitrary curve $(s,g(s),f(s)) \in C^{3}(I)$. We suppose that $g'' >0$ on $I$. Let $a(s) \in C^{1}(J)$ be a function such that $a'(s) <0$ on $J$, where $J=[s_{0},s_{1}]$ is a subinterval of $I$. Assume that $a(s_{0})<s_{0}$ and $[a(s_{1}),a(s_{0})] \subset I$. Consider a set of open segments $\Theta_{\operatorname{cup}}(J,g)$ consisting of those segments $\ell(s,g(s)), s \in J$ such that $\ell(s,g(s))$ is a segment  in the plane joining the  points $(s,g(s))$ and $(a(s),g(a(s)))$ (see Figure~\ref{fig:cup}).
\begin{Le}\label{cupfoliation}
The set of segments $\Theta_{\operatorname{cup}}(J,g)$ described above forms a foliation. 
\end{Le}
\begin{proof}
We need to check the 6 requirements for a set to be the foliation. Most of them are trivial except  for 4 and 5. We know the endpoints of each segment therefore we can consider the following argument function 
\begin{align*}
\theta(s) = \pi + \arctan \left(\frac{g(s)-g(a(s))}{s-a(s)}\right).
\end{align*}
Surely $\theta(s) \in C^{1}(J)$, so requirement 4 is satisfied. We check requirement 5. It is clear that it is enough to check this requirement for $x= (a(s),g(a(s))$. Let $s=s(x)$, then
\begin{align*}
&K(s) + \theta'(s)\|(a(s)-s,g(a(s))-g(s))\| =\frac{\langle (1, g'), (g-g(a),a-s)\rangle }{\| (g(a)-g,s-a)\|}+\\
&\frac{(g' - a'g'(a))(s-a)-(1-a')(g-g(a))}{\| (g(a)-g,s-a)\|} = \frac{a' \cdot \langle (1,g'(a)), (g-g(a),a-s)\rangle }{\| (g(a)-g,s-a)\|}
\end{align*}
which is strictly negative.
\end{proof}

Let $\gamma(t) = (t,g(t),f(t)) \in C^{3}([a_{0},b_{0}])$ be an arbitrary curve such that $g''>0$ on $[a_{0},b_{0}]$. Assume that the torsion of $\gamma$ is positive on $I^{-}=(a_{0},c)$, and it is negative on $I^{+}=(c,b_{0})$ for some $c\in (a_{0},b_{0})$.
\begin{Le}\label{existence}
For all $P$ such that $0<P<\min\{c-a_{0},b_{0}-c \}$ there exist $a \in I^{-}$, $b\in I^{+}$ such that $b-a=P$ and 
\begin{align}\label{cupeq0}
\begin{vmatrix}
1& 1 & a-b \\
g'(a)& g'(b) & g(a)-g(b) \\
f'(a)    & f'(b)     & f(a)-f(b) \\
\end{vmatrix} 
=0.
\end{align}
\end{Le}

 \begin{proof}
 Pick a number $a \in (a_{0},b_{0})$ so that $b=a+P \in (a_{0},b_{0})$.
We denote 
\begin{align*}
\mathcal{M}(a,b)  = (a-b)(g'(b)-g'(a))\left(\frac{g(a)-g(b)}{a-b}-g'(a)\right).
\end{align*}
 Note that the conditions  $a \neq b$ and $g''>0$ imply $\mathcal{M}(a,b) \neq 0$.
Then 
\begin{align*}
&
\begin{vmatrix}
1& 1 & a-b \\
g'(a)& g'(b) & g(a)-g(b) \\
f'(a)    & f'(b)     & f(a)-f(b) \\
\end{vmatrix}=
\mathcal{M}(a,b)\left[\frac{f(a)-f(b)-f'(a)(a-b)}{g(a)-g(b)-g'(a)(a-b)}-\frac{f'(b)-f'(a)}{g'(b)-g'(a)} \right].
\end{align*}
Thus our equation (\ref{cupeq0}) turns into

\begin{align}\label{cupeq}
\frac{f(a)-f(b)-f'(a)(a-b)}{g(a)-g(b)-g'(a)(a-b)}-\frac{f'(b)-f'(a)}{g'(b)-g'(a)}=0.
\end{align} 
We consider the following  functions $V(x) = f(x)-f'(a)x$ and $U(x) = g(x) - g'(a)x$. Note that $U(a)\neq U(b)$ and $U' \neq 0$ on $(a,b)$. Therefore by Cauchy's mean value theorem there exists a point $\xi=\xi(a,b) \in (a,b)$ such that
\begin{align*}
&\frac{f(a)-f(b)-f'(a)(a-b)}{g(a)-g(b)-g'(a)(a-b)} = \frac{V(a) - V(b)}{U(a)-U(b)} = \frac{V'(\xi)}{U'(\xi)} =\frac{f'(\xi)-f'(a)}{g'(\xi)-g'(a)}.
\end{align*}
Now we define 
\begin{align*}
W_{a}(z)  \df \frac{f'(z)-f'(a)}{g'(z)-g'(a)}, \quad z \in (a,b].
\end{align*}
So the left hand side of (\ref{cupeq}) takes the form $W_{a}(\xi) - W_{a}(b)=0$ for some $\xi(a,P) \in (a,b)$. 
We consider the curve $v(s) = (g'(s),f'(s))$ which is a graph on $[a_{0},b_{0}]$. The fact that the torsion of the curve $\gamma(s) = (s,g(s),f(s))$ changes sign from + to $-$ at the point $c \in (a_{0},b_{0})$ means that the curve $v(s)$ is strictly convex on the interval $(a_{0},c)$, and it is strictly concave on the interval $(c,b_{0})$.
We consider a function obtained from (\ref{cupeq})
\begin{align}\label{uroven}
D(z) \df \frac{f(z)-f(z+P)+f'(z) P}{g(z)-g(z+P)+g'(z) P}-\frac{f'(z+P)-f'(z)}{g'(z+P)-g'(z)},\quad z \in [a_{0},c].
\end{align}
Note that $D(a_{0}) = W_{a_{0}}(\zeta) - W_{a_{0}}(a_{0}+P)$ for some $\zeta=\zeta(a_{0},P) \in (a_{0},a_{0}+P)$. We know that $v(s)$ is strictly convex on the interval $(a_{0},a_{0}+P)$. This implies that $W_{a_{0}}(z) - W_{a_{0}}(a_{0}+P) <0$ for all $z \in (a_{0},a_{0}+P)$. In particular  $D(a_{0}) <0$. Similarly, concavity of $v(s)$ on $(c,c+P)$ implies that $D(c) >0$. Hence, there exists $a \in (a_{0},c)$ such that $D(a)=0$. 
\end{proof}
 Let $a_{1}$ and $b_{1}$ be  some solutions  of (\ref{cupeq0}) obtained by Lemma~\ref{existence}. 
\begin{Le}\label{functiona}
There exists a function $a(s) \in C^{1}((c,b_{1}])\cap C([c,b_{1}])$ such that $a(b_{1})=a_{1}$, $a(c)=c,$ $a'(s) <0$, and the pair $(a(s),s)$ solves the equation (\ref{cupeq0}) for all $s\in [c,b_{1}]$.
\end{Le}
\begin{proof}
The proof of the lemma is a consequence of the implicit function theorem. 
Let $a<b$, and consider the function 
\begin{align*}
\Phi(a,b) \df 
\begin{vmatrix}
1& 1 & a-b \\
g'(a)& g'(b) & g(a)-g(b) \\
f'(a)    & f'(b)     & f(a)-f(b) \\
\end{vmatrix}.
\end{align*}
We are going to find the signs of the partial derivatives of $\Phi(a,b)$ at the point $(a,b)=(a_{1},b_{1})$. We present the calculation only for $\partial \Phi/\partial b$. The case for $\partial \Phi/\partial a$ is similar. 

\begin{align*}
&\frac{\partial \Phi(a,b)}{\partial b}=
\begin{vmatrix}
1& 0 & a-b \\
g'(a)& g''(b) & g(a)-g(b) \\
f'(a) &f''(b)  & f(a)-f(b) \\
\end{vmatrix}
=
\\
&=(a-b)g''(b)\left(\frac{g(a)-g(b)}{a-b}-g'(a)\right)\left[\frac{f(a)-f(b)-f'(a)(a-b)}{g(a)-g(b)-g'(a)(a-b)}-\frac{f''(b)}{g''(b)} \right].
\end{align*}
 
Note that
\begin{align*}
(a-b)g''(b)\left(\frac{g(a)-g(b)}{a-b}-g'(a)\right) <0,
\end{align*}
therefore  we see that the sign of $\partial \Phi/\partial b$ depends only on the sign of the expression 
 
\begin{align}\label{difer1}
\frac{f(a)-f(b)-f'(a)(a-b)}{g(a)-g(b)-g'(a)(a-b)}-\frac{f''(b)}{g''(b)}.
\end{align}
 
We use the {\em cup equation} (\ref{cupeq}), and we obtain that the expression (\ref{difer1}) at the point $(a,b)=(a_{1},b_{1})$ takes the following form:

\begin{align}\label{differential}
\frac{f'(b)-f'(a)}{g'(b)-g'(a)}-\frac{f''(b)}{g''(b)}.
\end{align} 
The above expression has the following geometric meaning. We consider the curve $v(s)=(g'(s),f'(s))$, and we draw a segment which connects the points $v(a)$ and $v(b)$. The above expression is the difference between the slope of the line which passes through the segment $[v(a),v(b)]$ and the slope of the  tangent line of the curve $v(s)$ at the point $b$.
 In the case as it is shown on Figure~\ref{fig:slopes}, this difference is positive. Recall that $v(s)$ is strictly convex on $(a_{1},c)$, and it is strictly concave on $(c,b_{1})$. Therefore,  one can easily note that this expression (\ref{differential}) is always positive if the segment $[v(a),v(b)]$ also intersects the curve $v(s)$ at a point $\xi$ such that $a<\xi<b$. This always happens in our case because equation (\ref{cupeq}) means that the points $v(a),v(\xi),v(b)$ lie on  the same line, where $\xi$ was determined from  Cauchy's mean value theorem. Thus 
 \begin{align}\label{differentialnuzhen}
\frac{f'(b)-f'(a)}{g'(b)-g'(a)}-\frac{f''(b)}{g''(b)} >0.
\end{align} 

 Similarly, we can obtain that $\frac{\partial \Phi}{\partial a} <0$, because this is the same as to show that 
\begin{align}\label{differentiala}
\frac{f'(b)-f'(a)}{g'(b)-g'(a)}-\frac{f''(a)}{g''(a)}> 0.
\end{align}
 Thus, by the implicit function theorem there exists a $C^{1}$ function  $a(s)$  in some neighborhood of $b_{1}$ such that $a'(s) = -\frac{\Phi'_{b}}{\Phi'_{a}} <0$, and the pair $(a(s),s)$ solves (\ref{cupeq0}).
 
 Now we want to explain that the function $a(s)$ can be defined on $(c,b_{1}]$, and, moreover, $\lim_{s \to c+0}a(s) = c$.  Indeed, 
 whenever $a(s) \in (a_{1},c)$ and $s \in (c,b_{1})$ we can use the implicit function theorem, and we can extend the function $a(s)$. It is clear that for each $s$ we have $a(s) \in [a_{1},c)$ and $s \in (c,b_{1})$. Indeed, if  $a(s),s \in (a_{1},c]$, or $a(s),s \in [c,b_{1})$ then  (\ref{cupeq0}) has a definite sign (see (\ref{uroven})). It follows that $\alpha(s) \in C^{1}((c,b_{1}])$, and  the condition $a'(s)<0$ implies  $\lim_{s \to c+0}a(s) = c$. Hence $a(s) \in C([c,b_{1}]).$
 \end{proof}
 It is worth mentioning that we did not use the fact that the torsion of $(s,g(s),f(s))$ changes sign from + to $-$.  The only thing we needed was that the torsion changes sign. 
 
Let $a_{1}$ and $b_{1}$ be  any solutions of equation (\ref{cupeq0}) from Lemma~\ref{existence}, and let $a(s)$ be any function from Lemma~\ref{functiona}. Fix an arbitrary $s_{1} \in (c,b_{1})$ and consider the foliation $\Theta_{\operatorname{cup}}([s_{1},b_{1}],g)$ constructed by $a(s)$ (see Lemma~\ref{cupfoliation}). Let $\Bell$ be a function defined by (\ref{bellmanfunction}), where 
\begin{align}\label{initiald}
t_{2}(s_{1})=\frac{f'(s_{1})-f'(a(s_{1}))}{g'(s_{1})-g'(a(s_{1}))}.
\end{align}
 Set $\Omega_{\operatorname{cup}}=\Omega(\Theta_{\operatorname{cup}}([s_{1},b_{1}],g)),$ and let $\overline{\Omega_{\operatorname{cup}}}$ be the closure of $\Omega_{\operatorname{cup}}$.
\begin{Le}\label{bellmancup}
The  function $\Bell$ satisfies the following properties
\begin{itemize}
\item[1.] $\Bell \in C^{2}(\Omega_{\operatorname{cup}})\cap C^{1}(\overline{\Omega_{\operatorname{cup}}}) $.
\item[2.] $\Bell(a(s),g(a(s))) = f(a(s))$ for all $s\in [s_{1},b_{1}]$. 
\item[3.] $\Bell$ is a concave function in $\overline{\Omega_{\operatorname{cup}}}$.
\end{itemize}
\end{Le}
 \begin{proof}
 The first property follows from Lemma~\ref{dev} and the fact that $\nabla B(x) = t(s)$ for $s=s(x)$, where $s(x)$ is a continuous function in $\overline{\Omega_{\operatorname{cup}}}$. 
 
 We are going to check the second property. We recall (see (\ref{bdcondition})) that $t_{1}(s) = f'(s)-t_{2}(s)g'(s)$. Condition (\ref{initiald}) implies that 
 \begin{align}\label{initialed}
 t_{1}(s_{1})+t_{2}(s_{1})g'(a(s_{1}))=f'(a(s_{1})).
  \end{align}
  Let $\Bell(a(s),g(a(s)))=\tilde f (a(s))$.  After differentiation of this equality we get  $t_{1}(s_{1})+t_{2}(s_{1})g'(a(s_{1}))=\tilde f '(a(s_{1}))$. Hence, (\ref{initialed}) implies that $f'(a(s_{1})) = \tilde f '(a(s_{1}))$. It is clear that
  \begin{align*}
  &t_{1}(s) + t_{2}(s)g'(s) = f'(s),\\
  &t_{1}(s)+t_{2}(s)g'(a(s))=\tilde f' (a(s)),\\
  &t_{1}(s)(s-a(s))+t_{2}(s)(g(s)-g(a(s)))=f(s)-\tilde f (a(s)),
  \end{align*}
  which implies 
  \begin{align*}
  \begin{vmatrix}
1& 1 & s-a(s) \\
g'(s)& g'(a(s)) & g(s)-g(a(s)) \\
f'(s)    & \tilde f '(a(s))     & f(s)-\tilde f (a(s)) \\
\end{vmatrix}=0.
  \end{align*}
  This equality can be rewritten as follows:
  \begin{align*}
&  f' \cdot 
    \begin{vmatrix}
1 & s-a(s) \\
g'(a(s)) & g(s)-g(a(s)) \\
\end{vmatrix}
-\tilde f ' (a)
    \begin{vmatrix}
1 & s-a(s) \\
g' & g(s)-g(a(s)) \\
\end{vmatrix}+
(f-\tilde f (a))(g'(a(s))-g'(s))=0.
  \end{align*}
  By virtue of Lemma~\ref{functiona} we have the same equality as above except $\tilde f$ is replaced by $f$. We subtract one from another one:
\begin{align*}
  [f(a(s))-\tilde f (a(s))] +[f'(a(s))-\tilde f '(a(s)) ] \cdot 
  \frac{   \begin{vmatrix}
1 & s-a(s) \\
g' & g(s)-g(a(s)) \\
\end{vmatrix}
  }{g'(a(s))-g'(s)}=0. 
\end{align*} 
Note that 
\begin{align*}
 \frac{   \begin{vmatrix}
1 & s-a(s) \\
g' & g(s)-g(a(s)) \\
\end{vmatrix}
  }{g'(a(s))-g'(s)}<0
 \end{align*}
 and $a(s)$ is invertible. Therefore we get the  differential equation $z(u)B(u)+z'(u)=0$ where $B \in C^{1}([a(b_{1}),a(s_{1})])$, $z(u) = f(u)-\tilde f (u)$ and $B <0$. The condition $z'(a(s_{1}))=0$ implies $z(a(s_{1}))=0$. Note that $z=0$ is a trivial solution. Therefore, by uniqueness of solutions to  ODEs we get $z=0$.  
 
 We are going to check the concavity of $\Bell$. Let $\mathcal{F}$ be the force function corresponding to $\Bell$. By Corollary~\ref{torsion} we only need to check that $\mathcal{F}(s_{1})\leq 0$. Note that (\ref{silfun}) and  (\ref{initiald}) imply 
 \begin{align*}
 \mathcal{F}(s_{1})=\frac{f''(s_{1})}{g''(s_{1})}-t_{2}(s_{1})=\frac{f''(s_{1})}{g''(s_{1})}-\frac{f'(s_{1})-f'(a(s_{1}))}{g'(s_{1})-g'(a(s_{1}))},
 \end{align*}
 which is negative by (\ref{differentialnuzhen}).
 \end{proof}
  \begin{Rem}
  The above lemma is true for all choices $s_{1} \in (c,b_{1})$. If we send $s_{1}$ to $c$ then one can easily see that $\lim_{s_{1}\to c+} t_{2}(s_{1})=0$, therefore the force function $\mathcal{F}$ takes the following form
  \begin{align*}
  \mathcal{F}(s) = \int_{c}^{s} \left[\frac{f''(y)}{g''(y)}\right]' \exp\left( -\int_{y}^{s} \frac{g''(r)}{K(r)}\cos \theta (r)dr\right)dy.
  \end{align*}
This is another way to show that the force function is nonpositive.  
 \end{Rem}
 \begin{wrapfigure}[12]{r}{0pt}
\includegraphics[scale=0.78]{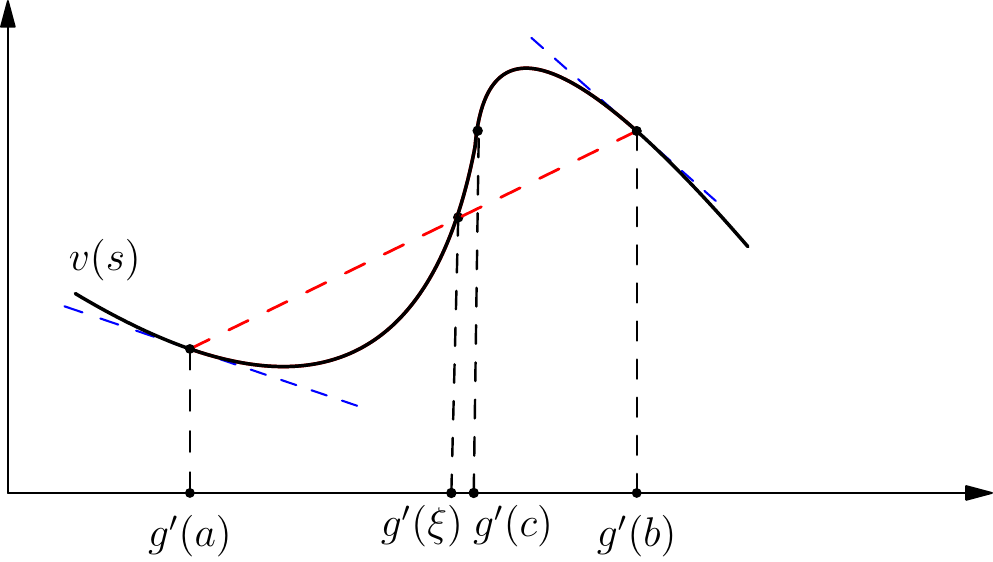}
\caption{Graph $v(s)$}
\label{fig:slopes}
\end{wrapfigure}

The next lemma  shows that the regardless of the choices of initial solution $(a_{1},b_{1})$ of (\ref{cupeq0}), the constructed function $a(s)$ by Lemma~\ref{functiona} is unique (i.e. it does not depend on the pair $(a_{1},b_{1})$).

\begin{Le}\label{unique}
Let pairs $(a_{1},b_{1})$, $(\tilde a_{1} , \tilde b_{1} )$ solve (\ref{cupeq0}), and let $a(s), \tilde a (s)$ be the corresponding functions obtained by Lemma~\ref{functiona}. Then $a(s) = \tilde a (s)$ on $[c, \min\{ b_{1}, \tilde b_{1} \}]$.
\end{Le}

\begin{proof}
By the uniqueness result of the implicit function theorem we only need to show existence of $s_{1} \in (c, \min\{ b_{1}, \tilde b_{1} \})$ such that $a (s_{1}) = \tilde a (s_{1})$. Without loss of generality assume that $\tilde b_{1}  = b_{1}=s_{2}$. We can also assume that $\tilde a (s_{2}) > a (s_{2})$, because other cases can be solved in a similar way. 

Let $\Theta = \Theta_{\operatorname{cup}}([c,s_{2}],g)$ and $\tilde{\Theta} = \tilde{\Theta}_{\operatorname{cup}}([c,s_{2}],g)$ be the foliations corresponding to the functions $a(s)$ and $\tilde a (s)$. Let $\Bell$ and $\tilde \Bell$ be the functions corresponding to these foliations from Lemma~\ref{bellmancup}. We consider a chord $T$  in $\mathbb{R}^{3}$ joining the points $(a(s_{1}),g(a(s_{1})),f(a(s_{1})))$  and $(s_{1},g(s_{1}),f(s_{1}))$ (see Figure~\ref{fig:uniqueness}). We want to show that the chord $T$ belongs to the graph of $\tilde \Bell$.
Indeed,   concavity of $\tilde{\Bell}$ (see Lemma~\ref{bellmancup}) implies that the chord $T$ lies below the graph of  $\tilde{\Bell}(x_{1},x_{2})$, where $(x_{1},x_{2}) \in \Omega(\tilde{\Theta})$. Moreover, concavity of $\Bell$, $\Omega(\tilde{\Theta})\subset \Omega(\Theta)$ and the fact that the graph $\tilde \Bell$ consists of chords joining the points of the curve $(t,g(t),f(t))$ imply that  the graph $\Bell$ lies above the graph $\tilde{\Bell}$. In particular the chord $T$, belonging to the graph $\Bell$, lies above the graph $\tilde{\Bell}$. This can happen if and only if  $T$ belongs to the graph $\tilde{\Bell}$.
Now we show that if $s_{1}<s_{2}$, then the torsion of the curve $(s,g(s),f(s))$ is zero for $s \in [s_{1},s_{2}]$. Indeed, 
 let $\tilde T$ be a chord in $\mathbb{R}^{3}$ which joins the points $(a(s_{1}),g(a(s_{1})),f(a(s_{1})))$ and $(s_{2},g(s_{2}),f(s_{2}))$. We consider the tangent plane $L(x)$  to the graph  $\tilde \Bell$ at the point $(x_{1},x_{2}) = (a(s_{1}),g(a(s_{1})))$. This tangent plane must contain both chords $T$ and $\tilde T$, and it must be  tangent to the surface at these chords. Concavity of $\tilde \Bell$ implies that the tangent plane $L$ coincides with $\tilde \Bell$ at points belonging to the triangle, which is the convex hull of the points $(a(s_{1}), g(a(s_{1})))$, $(s_{1},g(s_{1}))$ and $(s_{2},g(s_{2}))$. Therefore, it is clear that  the tangent plane $L$ coincides with $\tilde \Bell$ on the segments $\ell \in \tilde{\Theta}$ with the endpoint at $(s,g(s))$ for $s \in [s_{1},s_{2}]$. Thus $L((s,g(s)))=\tilde \Bell((s,g(s)))$  for any $s \in [s_{1},s_{2}].$ This means that the torsion of the curve $(s,g(s),f(s))$ is zero  on $s \in [s_{1},s_{2}]$ which contradicts our assumption about the torsion. Therefore $s_{1}=s_{2}$. 
\end{proof}

\begin{Cor}\label{edin}
In the conditions of Lemma~\ref{existence},  for all $0<P<\min\{c-a_{0},b_{0}-c\}$ there exists a unique pair $(a_{1},b_{1})$  which solves (\ref{cupeq0}) such that  $b_{1}-a_{1}=P$.
\end{Cor}

The above corollary implies that if the pairs $(a_{1},b_{1})$ and $(\tilde a_{1}, \tilde b_{1} )$ solve (\ref{cupeq0}), then $a_{1} \neq \tilde a _{1}$ and $b_{1} \neq \tilde b _{1}$, and one of the following conditions holds: $(a_{1},b_{1})\subset (\tilde{a}_{1}, \tilde{b}_{1})$, or $(\tilde{a}_{1},\tilde{b}_{1})\subset (a_{1}, b_{1})$.

\begin{Rem}\label{extenda}
The function $a(s)$ is defined on the right of the point $c$. We extend naturally its definition on the left  of the interval by $a(s) \df a^{-1}(s)$.
\end{Rem}

\section{Construction of the Bellman function}\label{construction}
\subsection{Reduction to the two dimensional case}\label{redu}
We are going to construct the Bellman function for the case $p<2$. The case $p=2$ is trivial, and the case $p>2$ was solved in~\cite{BJV}. 
From the definition of  $H$ it follows that 
\begin{align}\label{symmetry}
H(x_{1},x_{2},x_{3})=H(|x_{1}|,|x_{2}|,x_{3}) \quad  \text{for all} \quad  (x_{1},x_{2},x_{3}) \in \Omega.
\end{align}
Also note the homogeneity condition 
\begin{align}\label{homo}
H(\lambda x_{1},\lambda x_{2}, \lambda^{p} x_{3} )=\lambda^{p} H(x_{1},x_{2},x_{3}) \quad  \text{for all} \quad  \lambda \geq 0.
\end{align}
These two conditions (\ref{symmetry}), (\ref{homo}), which follow from the nature of the boundary data $(x^{2}+\tau^{2}y^{2})^{2/p}$, make the construction of $H$ easier. However, in order to construct the function $H$, this information is not necessary.
Further, we assume that $H$ is $C^{1}(\Omega)$ smooth. Then from the symmetry (\ref{symmetry}) it follows that 
\begin{align}\label{neiman0}
\frac{\partial H}{\partial x_{j}}=0  \quad \text{on} \quad x_{j}=0 \quad \text{for}\quad  j=1,2.   
\end{align}
For convenience, as in~\cite{BJV}, we rotate the system of coordinates $(x_{1},x_{2},x_{3})$. Namely, let 
\begin{align}\label{coorchange}
y_{1} \df \frac{x_{1}+x_{2}}{2},\quad y_{2} \df \frac{x_{2}-x_{1}}{2}, \quad  y_{3}\df x_{3}.
\end{align}
 We define 
\begin{align*}
N(y_{1},y_{2},y_{3}) \df H(y_{1}-y_{2},y_{1}+y_{2},y_{3}) \quad \text{on} \quad \Omega_{1},
\end{align*}
where $\Omega_{1} = \{(y_{1},y_{2},y_{3})\colon y_{3} \geq 0, \,|y_{1}-y_{2}|^{p} \leq y_{3} \}$. It is clear that for fixed $y_{1}$, the function $N$ is concave in variables $y_{2}$ and $y_{3}$; moreover, for fixed $y_{2}$ the function $N$ is concave with respect to the rest of variables. 
The symmetry (\ref{symmetry}) for $N$ turns into the following condition
\begin{align}\label{symmetrym}
N(y_{1},y_{2},y_{3})=N(y_{2},y_{1},y_{3})=N(-y_{1},-y_{2},y_{3}).
\end{align}
Thus it is sufficient to construct the function $N$ on the domain
\begin{align*}
\Omega_{2} \df \{ (y_{1},y_{2},y_{3})\colon y_{1} \geq 0, \; -y_{1} \leq y_{2} \leq y_{1}, \; (y_{1}-y_{2})^{p} \leq y_{3} \}.
\end{align*}
Condition (\ref{neiman0}) turns into
\begin{align}
\frac{\partial N}{\partial y_{1}} &= \frac{\partial N}{ \partial y_{2}} \quad \text{on the hyperplane} \quad y_{2}=y_{1},  \label{ne1}\\
\frac{\partial N}{\partial y_{1}} &= -\frac{\partial N}{ \partial y_{2}} \quad \text{on the hyperplane} \quad y_{2}=-y_{1}. \label{ne2}
\end{align}
The boundary condition (\ref{boundarycondition}) becomes
\begin{align}\label{bmcondition}
N(y_{1},y_{2},|y_{1}-y_{2}|^{p}) = ((y_{1}+y_{2})^{2}+\tau^{2} (y_{1}-y_{2})^{2})^{p/2}.
\end{align}
The homogeneity condition (\ref{homo}) implies that $N(\lambda y_{1},\lambda y_{2},\lambda^{p} y_{3}) = \lambda^{p} N(y_{1},y_{2},y_{3})$ for $\lambda  \geq 0$. We choose $\lambda  = 1/y_{1}$, and we obtain that 
\begin{align}\label{extend}
N(y_{1},y_{2},y_{3}) =y_{1}^{p} N\left(1,\frac{y_{2}}{y_{1}},\frac{y_{3}}{y_{1}^{p}}\right)
\end{align}  
Suppose we are able to construct the function $M(y_{2},y_{3}) \df N(1,y_{2},y_{3})$ on 
\begin{align*}
\Omega_{3} \df \{ (y_{2},y_{3}):\;  -1 \leq y_{2} \leq 1, (1-y_{2})^{p} \leq y_{3} \}
\end{align*}
with the following conditions: 
\begin{itemize}
\item[1.] $M$ is concave in $\Omega_{3}$ 
\item[2.] $M$ satisfies (\ref{bmcondition}) for $y_{1}=1$.
\item[3.] The extension of $M$ onto  $\Omega_{1}$ via formulas (\ref{extend}) and (\ref{symmetrym}) is a function with the properties of $N$ (see (\ref{ne1}), (\ref{ne2}), and concavity of $N$).
\item[4.] $M$ is minimal among those who satisfy the conditions 1,2,3. 
\end{itemize}
Then the extended function $M$ should be $N$.
 So we are going to construct  $M$ on $\Omega_{3}$.
 We denote
\begin{align}
&g(t) \df (1-t)^{p}, \quad t \in [-1,1], \label{gc}\\
&f(t) \df ((1+t)^{2} + \tau^{2} (1-t)^{2})^{p/2}, \quad t \in [-1,1]. \label{fc}
\end{align}
Then we have the boundary condition 
\begin{align}\label{bc}
M(t,g(t)) =f(t), \quad  t \in [-1,1].
\end{align} 

We differentiate the condition (\ref{extend}) with respect to $y_{1}$ at the point $(y_{1},y_{2},y_{3}) = (1,-1,y_{3})$ and we obtain that 
\begin{align*}
\frac{\partial N}{\partial y_{1}}(1,-1,y_{3}) = pN(1,-1,y_{3})+\frac{\partial N}{\partial y_{2}}(1,-1,y_{3})-py_{3}\frac{\partial N}{\partial y_{3}}, \quad y_{3} \geq 0.
\end{align*}
Now we use (\ref{ne2}), so we obtain another requirement for $M(y_{2},y_{3})$:
\begin{align}\label{req1}
0=pM(-1,y_{3}) + 2\frac{\partial M}{\partial y_{2}}(-1,y_{3})-py_{3}\frac{\partial M}{\partial y_{3}}(-1,y_{3}), \quad \text{for} \quad  y_{3} \geq 0.
\end{align}
Similarly, we differentiate  (\ref{extend}) with respect to $y_{1}$ at point $(y_{1},y_{2},y_{3}) = (1,1,y_{3})$ and use (\ref{ne1}), so we obtain
\begin{align}\label{req2}
0=pM(1,y_{3}) - 2\frac{\partial M}{\partial y_{2}}(1,y_{3})-py_{3}\frac{\partial M}{\partial y_{3}}(1,y_{3}), \quad \text{for} \quad  y_{3} \geq 0.
\end{align}
So in order to satisfy conditions (\ref{ne1}) and (\ref{ne2}), the requirements (\ref{req1}) and (\ref{req2}) are necessary. It is easy to see that these requirements are also sufficient in order to satisfy these conditions.

\begin{wrapfigure}[12]{r}{0pt}
\includegraphics[scale=0.78]{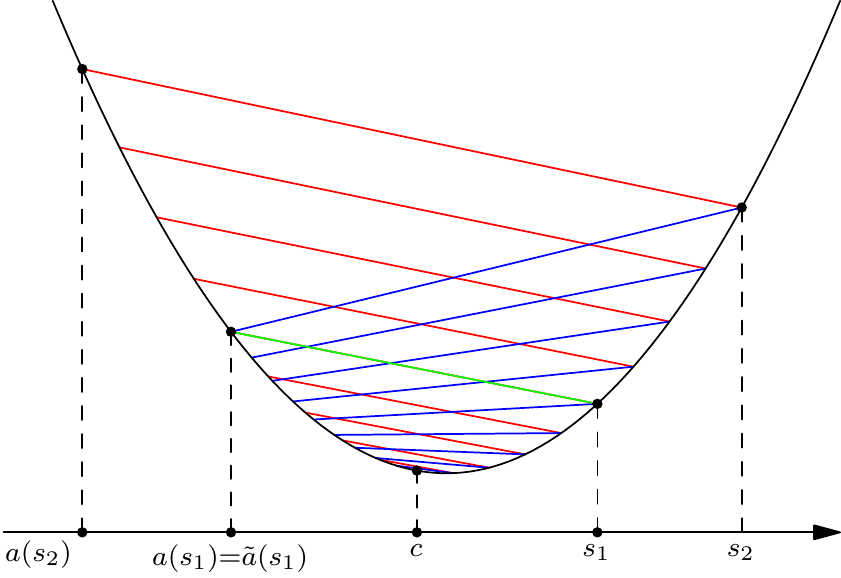}
\caption{Uniqueness of the cup}
\label{fig:uniqueness}
\end{wrapfigure}
The minimum between two concave functions with fixed boundary data is a concave function with the same boundary data. Note also that the conditions (\ref{req1}) and (\ref{req2}) still fulfilled  after taking the minimum. Thus it is quite reasonable to construct a candidate for  $M(y_{2},y_{3})$ as a minimal concave function on $\Omega_{3}$ with the boundary conditions (\ref{bc}), (\ref{req1}) and (\ref{req2}). We remind that we should also have the concavity of the extended function $N(y_{1},y_{2},y_{3})$ with respect to variables $y_{1},y_{3}$ for each fixed $y_{2}$. This condition can be verified after the construction of the function $M(y_{2},y_{3})$.

\subsection{Construction of a candidate for M}
We are going to construct a candidate $\s{B}$ for $M$. 
Firstly, we show that for $\tau >0$, the torsion $\tau_{\gamma}$ of the boundary curve $\gamma(t) \df (t,g(t),f(t))$ on $t \in (-1,1)$, where $f,g$ are defined by (\ref{gc}) and (\ref{fc}), changes sign once from + to $-$. We call this point the root of a cup. We construct the cup around this point. 
Note that $g'<0, g''>0$ on $[-1,1).$ Therefore  
\begin{align*}
\sign \tau_{\gamma} = \sign\left(f''' - \frac{g'''}{g''}f'' \right)=\sign \left( f''' - \frac{2-p}{1-t}f''\right)=
\sign(v(t)),
\end{align*}
where 
\begin{align*}
&v(t)\df -(1+\tau^{2})^{2}(p-1)t^{3}+(1+\tau^{2})(3\tau^{2}+\tau^{2}p+3-3p)t^{2}+\\
&(2\tau^{2}p-9\tau^{4}+\tau^{4}p+3-3p-6\tau^{2})t-p+5\tau^{4}+2\tau^{2}p-\tau^{4}p-10\tau^{2}+1.
\end{align*}
Note that $v(-1) = 16\tau^{4}>0$ and $v(1)=-8((p-1)+\tau^{2})<0$. So the function $v(t)$ changes sign from + to $-$ at least once. Now, we show that $v(t)$ has only one root. For $\tau^{2} <  \frac{3(p-1)}{3-p}$, note that the linear function $v''(t)$ is nonnegative i.e.  $v''(-1) =8\tau^{2}p(1+\tau^{2})>0$,  $v''(1) = -4(1+\tau^{2})(\tau^{2}p-3\tau^{2}+3p-3)\geq 0$. Therefore,  the convexity of $v(t)$ implies the uniqueness of the root $v(t)$ on $[-1,1]$.

 Suppose $\tau^{2} <  \frac{3(p-1)}{3-p}$;  we will show that $v' \leq 0$ on $[-1,1]$. Indeed,  the discriminant of the quadratic function $v'(x)$ has the expression 
\begin{align*}
D = 16\tau^{2}(\tau^{2}+1)^{2}((3-p)^{2}\tau^{2}-9(p-1)),
\end{align*}
 which is negative for $0<\tau^{2} <\frac{3(p-1)}{3-p}$. Moreover, $v'(-1) = -4\tau^{2}(\tau^{2}p+3\tau^{2}+3)<0$. Thus we obtain that $v'$ is negative. 
 
 We denote the root of $v$ by $c$.
 It is an appropriate time to make the following remark.
 \begin{Rem}\label{poryadok}
 Note that $v(-1+2/p)<0.$ Indeed, 
 \begin{align*}
 v(-1+2/p) = \frac{(3p-2)(p^{2}-2p-4)\tau^{4}+(16+5p^{3}-8p^{2}-16p)\tau^{2}+8(1-p)}{p^{3}},
 \end{align*}
 which is negative because coefficients of $\tau^{4}, \tau^{2}, \tau^{0}$ are negative. Therefore, this inequality implies that $c < -1+2/p$. 
 \end{Rem}

  Consider  $a=-1$ and $b=1$;   the left side of  (\ref{cupeq0}) takes the positive value $-2^{2p-1}p(1-p)$. However, if we consider  $a=-1$ and $b=c$, then the proof of Lemma~\ref{existence} (see (\ref{uroven})) implies that the left side of (\ref{cupeq0}) is negative. Therefore, there exists a unique $s_{0} \in (c,1)$ such that the pair $(-1,s_{0})$ solves (\ref{cupeq0}). Uniqueness follows from Corollary~\ref{edin}. The equation (\ref{cupeq0}) for the pair $(-1,s_{0})$ is equivalent to the equation $u\left( \frac{1+s_{0}}{1-s_{0}}\right)=0,$ where 
 \begin{align}\label{lkonec}
u(z) \df \tau^{p}(p-1)\left(\tau^{2}+z^{2}\right)^{(2-p)/2}-\tau^{2}(p-1)+(1+z)^{2-p}-z(2-p)-1.
\end{align}
Lemma~\ref{functiona} gives the function $a(s),$ and Lemma~\ref{bellmancup} gives the concave function $\Bell(y_{2},y_{3})$ for $s_{1}=c$ with the foliation $\Theta_{\operatorname{cup}}((c,s_{0}],g)$ in the domain $ \Omega(\Theta_{\operatorname{cup}}((c,s_{0}],g))$.
 
The above explanation implies the following corollary. 
\begin{Cor}\label{ravnosilno}
Pick any point $\tilde{y}_{2} \in (-1,1).$ The inequalities $s_{0} < \tilde{y}_{2}$, $s_{0}=\tilde{y}_{2}$ and $\tilde{y}_{2}>s_{0}$ are equivalent to the following inequalities respectively: $u\left(\frac{1+\tilde{y}_{2}}{1-\tilde{y}_{2}}\right)<0$, $u\left(\frac{1+\tilde{y}_{2}}{1-\tilde{y}_{2}}\right)=0$ and $u\left(\frac{1+\tilde{y}_{2}}{1-\tilde{y}_{2}}\right)>0$.
\end{Cor}

\begin{wrapfigure}[14]{r}{0pt}
\includegraphics[scale=0.78]{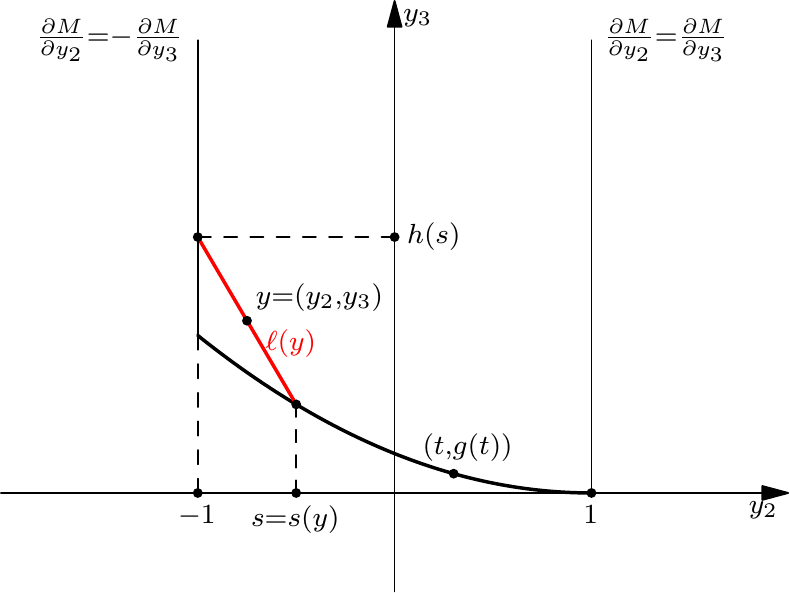}
\caption{Segment $\ell(y)$}
\label{fig:left}
\end{wrapfigure}
Now we are going to extend $C^{1}$ smoothly the function $\Bell$  on the upper part of the cup. Recall that we are looking for a minimal concave function. If we construct a function with a foliation $\Theta([s_{0}, \tilde y _{2}],g)$ where $\tilde y _{2} \in (s_{0},1)$ then the best thing we can do according to Lemma~\ref{gluing} and Lemma~\ref{comparison} is to minimize $\sin (\theta_{\operatorname{cup}}(s_{0}) -  \theta (s_{0}))$ where $\theta_{\operatorname{cup}}(s)$ is an argument function of $\Theta_{\operatorname{cup}}((c,s_{0}],g)$  and 
$\theta(s)$ is an argument function of $\Theta([s_{0}, \tilde y _{2}],g)$. In other words we need to choose segments from $\Theta([s_{0}, \tilde y _{2}],g)$ close enough to the segments of $\Theta_{\operatorname{cup}}((c,s_{0}],g)$. 
 
Thus,  we are going to try to construct the set of segments $\Theta([s_{0},\tilde y _{2}])$ so that  they start from $(s,g(s),f(s))$, $s\in [s_{0},\tilde y _{2}],$ and they go to the boundary $y_{2}=-1$ of $\Omega_{3}$.

 We explain how the conditions (\ref{req1}) and (\ref{req2})  allow us to construct such type of foliation $\Theta([s_{0},\tilde y _{2}],g)$ in a unique way. Let $\ell(y)$ be the segment with the endpoints $(s,g(s))$ where $s \in (s_{0},\tilde y _{2})$ and $(-1,h(s))$ (see Figure~\ref{fig:left}).

Let $t(s) = (t_{1}(s),t_{2}(s))=\nabla \Bell(y) $ where $s=s(y)$ is the corresponding gradient function. 
Then (\ref{req1}) takes the form
\begin{align}\label{neim1}
0=p\Bell(-1,h(s))+2t_{1}(s)-ph(s)t_{2}(s).
\end{align}
We differentiate this expression with respect to $s$, and we obtain  
\begin{align}\label{neimshtrix}
2t'_{1}(s)-ph(s)t'_{2}(s)=0.
\end{align}
Then according to (\ref{extremal}) we find the function $\tan 
\theta(s)$, and, hence, we find the quantity $h(s)$
\begin{align*}
\tan \theta(s) = -\frac{ph(s)}{2}\quad  \Leftrightarrow \quad \frac{h(s)-g(s)}{s+1}= \frac{ph(s)}{2}.
\end{align*}
Therefore, 
\begin{align}\label{aes}
h(s) = \frac{2g(s)}{p}\left(\frac{1}{y_{p}-s}\right) \quad  \text{where}\quad  y_{p} \df -1+\frac{2}{p}.
\end{align}
We see that the function $h(s)$ is well defined, it increases,  and it is differentiable on $-1 \leq s < y_{p}$. So we conclude that  if $s_{0} < y_{p}$ then we are able to construct the set of  segments $\Theta([s_{0},y_{p}),g)$   that pass through the points $(s,g(s))$ , where  $s \in [s_{0},y_{p})$ and through the boundary $y_{2}=-1$ (see Figure~\ref{fig:cupandchord}).

\begin{wrapfigure}[16]{r}{0pt}
\includegraphics[scale=0.9]{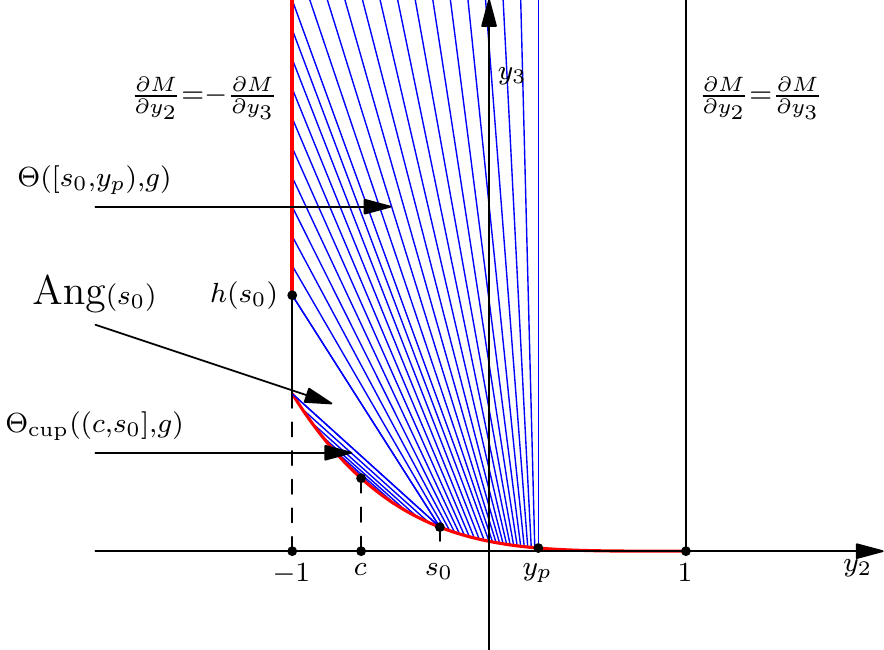}
\caption{Foliations $\Theta_{\operatorname{cup}}((c,s_{0}],g)$ and $\Theta([s_{0},y_{p}),g)$ }
\label{fig:cupandchord}
\end{wrapfigure}

 It is easy to check that $\Theta([s_{0},y_{p}),g)$ is a foliation. So choosing the value $t_{2}(s_{0})$ of $\Bell$ on $\Omega(\Theta([s_{0},y_{p}),g))$ according to  Lemma~\ref{gluing}, then  by Corollary~\ref{rassh} we have constructed the concave function $\Bell$ in the domain $\Omega(\Theta_{\operatorname{cup}}((c,s_{0}],g))\cup \operatorname{Ang}(s_{0})\cup \Omega(\Theta([s_{0},y_{p}],g))$. 

 It is clear that the foliation $\Theta([s_{0},y_{p}),g)$ exists as long as $s_{0}<y_{p}$.  
  Note that $\frac{1+y_{p}}{1-y_{p}}=\frac{1}{p-1}$. Therefore, Corollary~\ref{ravnosilno} implies the following remark.
 \begin{Rem}\label{daxmareba}
The  inequalities $s_{0}<y_{p},$ $s_{0}=y_{p}$ and $s_{0}>y_{p}$ are equivalent to the following inequalities respectively: $u\left(\frac{1}{p-1}\right) <0$, $u\left(\frac{1}{p-1}\right) =0$ and $u\left(\frac{1}{p-1}\right) >0$.
 \end{Rem}
   
\begin{wrapfigure}[16]{r}{0pt}
\includegraphics[scale=0.85]{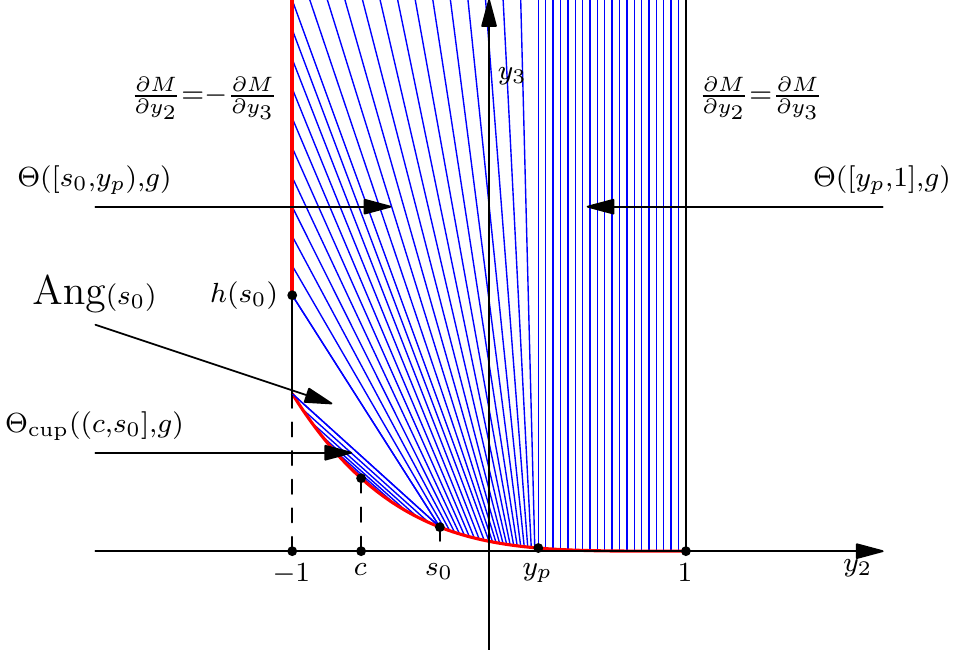}
\caption{Case $u\left(\frac{1}{p-1}\right)< 0$ }
\label{fig:ccv}
\end{wrapfigure}
  At the point $y_{p}$ the segments from $\Theta([s_{0},y_{p}),g)$ become vertical. After the point $(y_{p},g(y_{p}))$ we should consider vertical segments $\Theta([y_{p},1],g)$ (see Figure~\ref{fig:ccv}), because by Lemma~\ref{comparison} this corresponds  to the minimal function.  Surely $\Theta([y_{p},1],g)$ is the foliation. Again, choosing the value $t_{2}(y_{p})$ of $\Bell$ on $\Omega(\Theta([y_{p},1],g))$ according to  Lemma~\ref{gluing}, then  by Corollary~\ref{rassh} we have constructed the concave function $\Bell$  on $\Omega_{3}$. 
Note that if $s_{0}\geq y_{p}$ (which corresponds to the inequality $u\left(\frac{1}{p-1}\right) >0$) then we do not have the foliation $\Theta([s_{0},y_{p}),g)$. In this  case we consider only vertical segments $\Theta([s_{0},1],g)$ (see Figure~\ref{fig:vvv}),
and again choosing the value $t_{2}(s_{0})$ of $\Bell$ on $\Omega(\Theta([s_{0},1],g))$ according to  Lemma~\ref{gluing} then  by Corollary~\ref{rassh} we construct a  concave function $\Bell$  on $\Omega_{3}$. We believe that $\Bell = M$.

 We still have to check the requirements (\ref{req1}) and (\ref{req2}).  The crucial role is played by symmetry of the boundary data of $N$. Further, the given proofs work for both of the cases $y_{p}<s_{0}$ and $y_{p} \geq s_{0}$. Therefore, we do not consider them separately. 
 
  The requirement (\ref{req2}) follows immediately. Indeed, the condition (\ref{bellmanfunction}) at the point $y=(1,y_{3})$ (note that in (\ref{bellmanfunction}) instead of $x=(x_{1},x_{2})$ we consider $y=(y_{2},y_{3})$) implies that $\Bell(1,y_{3}) = f(1)+t_{2}(1)(y_{3}-g(1))$.  Therefore, the requirement (\ref{req2}) takes the form $0=pf(1)-2t_{1}(1)$. Using (\ref{bdcondition}), we obtain that $t_{1}(1)=f'(1)$. Therefore, we see that $pf(1)-2t_{1}(1) = pf(1)-2f'(1) = 0$.

Now, we are going to obtain the requirement (\ref{req1}) which is the same as (\ref{neim1}). The quantities $t_{1},t_{2}$ of $\Bell$ with the foliation $\Theta([s_{0},y_{p}),g)$ satisfy the condition (\ref{neimshtrix}) which was obtained by differentiation of (\ref{neim1}). So we only need to check the condition (\ref{neim1}) at the initial point $s=s_{0}$. If we substitute the expression of $\Bell$ from (\ref{bellmanfunction}) into (\ref{neim1}), then (\ref{neim1}) turns into the following equivalent condition: 
\begin{align}\label{neim3}
t_{1}(s)(s-y_{p})+t_{2}(s)g(s)=f(s).
\end{align}

Note that (\ref{bdcondition}) allows us to rewrite (\ref{neim3}) into the  equivalent condition 
\begin{align}\label{fromsurf1}
t_{2}(s) = \frac{f(s)-(s-y_{p})f'(s)}{g(s)-(s-y_{p})g'(s)} .
\end{align} 
And as it was mentioned above we only need to check condition (\ref{fromsurf1}) at the point $s=s_{0}$, i.e.
\begin{align}\label{fromsurf}
t_{2}(s_{0}) = \frac{f(s_{0})-(s_{0}-y_{p})f'(s_{0})}{g(s_{0})-(s_{0}-y_{p})g'(s_{0})} .
\end{align}

On the other hand, if we differentiate the boundary condition $\Bell(s,g(s))=f(s)$ at the points $s=s_{0}, -1,$ then we obtain 
\begin{align*}
t_{1}(s_{0})+t_{2}(s_{0})g'(-1)&=f'(-1),\\
t_{1}(s_{0})+t_{2}(s_{0})g'(s_{0})&=f'(s_{0}).\\
\end{align*}
Thus we can find the  value of $t_{2}(s_{0})$: 
\begin{align}\label{fromcup}
t_{2}(s_{0}) = \frac{f'(-1)-f'(s_{0})}{g'(-1)-g'(s_{0})}.
\end{align}
So these two values (\ref{fromcup}) and (\ref{fromsurf}) must coincide. In other words we need to show 
\begin{align}\label{coincide0}
\frac{f(s_{0})-(s_{0}-y_{p})f'(s_{0})}{g(s_{0})-(s_{0}-y_{p})g'(s_{0})}  =  \frac{f'(-1)-f'(s_{0})}{g'(-1)-g'(s_{0})}.
\end{align}
It will be convenient for us to work with the following notations for  the rest of the current subsection. We denote $g(-1)=g_{-}$, $g'(-1) = g'_{-}$, $f(-1)=f_{-}$, $f'(-1)=f'_{-}$
$g(s_{0})=g, g'(s_{0})=g', f(s_{0}) =f, f'(s_{0})=f'$. The condition (\ref{coincide0}) is equivalent to 
\begin{align}\label{odno}
s_{0} &= \frac{fg'_{-}+f'g-fg'-gf'_{-}}{f'g'_{-}-g'f'_{-}}+y_{p} = \\
  &=\left( \frac{fg'_{-}+f'g-fg'-gf'_{-}}{f'g'_{-}-g'f'_{-}} - 1 \right) +\frac{2}{p}. \nonumber
\end{align}
  
On the other hand, from (\ref{cupeq0}) for the pair $(-1,s_{0})$ we obtain that 
\begin{align*}
s_{0} = \left(\frac{fg'_{-}+f'g-fg'-gf'_{-}}{f'g'_{-}-g'f'_{-}}-1\right)+\frac{f'g_{-}+g'_{-}f_{-}-g'f_{-}-f'_{-}g_{-}}{g'f'_{-}-f'g'_{-}}.
\end{align*}

So, from (\ref{odno}) we see that it suffices to show that 

\begin{align*}
\frac{f'g_{-}+g'_{-}f_{-}-g'f_{-}-f'_{-}g_{-}}{g'f'_{-}-f'g'_{-}}=\frac{2}{p}. 
\end{align*}
We note that $g'_{-} = -(p/2)g_{-}, f'_{-} =-(p/2)f_{-}$, hence  $g'_{-}f_{-} = f'_{-}g_{-}$. Therefore, we have
\begin{align*}
\frac{f'g_{-}+g'_{-}f_{-}-g'f_{-}-f'_{-}g_{-}}{g'f'_{-}-f'g'_{-}} = \frac{f'g_{-}-g'f_{-}}{g'f'_{-}-f'g'_{-}} = \frac{2}{p}.
\end{align*}

\subsection{Concavity in another direction}

We are going to check the concavity of the extended function $N$ via $\Bell$ in another direction. It is worth mentioning that the both of the cases $y_{p}<s_{0}$, $y_{p}\geq s_{0}$ do not play any role in the following computations, therefore we consider them together. We define a candidate for $N$ as 
\begin{align}\label{extendab}
N(y_{1},y_{2},y_{3}) \df y_{1}^{p} \Bell(1,y_{2}/y_{1}, y_{3}/y_{1}^{p})\quad \text{for} \quad \left( \frac{y_{2}}{y_{1}}, \frac{y_{3}}{y_{1}^{p}}\right) \in \Omega_{3},
\end{align}
and we extend $N$ to the $\Omega_{1}$ by (\ref{symmetrym}). Then, as it was already discussed, $N \in C^{1}(\Omega_{1})$. 
 We need the following technical lemma:
\begin{Le}\label{anotherdirection}
\begin{align*}
N''_{y_{1}y_{1}}N''_{y_{3}y_{3}}-(N''_{y_{1}y_{3}})^{2} = -t'_{2}s'_{y_{3}}p(p-1)y_{1}^{p-2}\left(st_{1}+gt_{2}-f + \frac{y_{2}}{y_{1}}t_{1}\cdot \left(\frac{2}{p}-1\right) \right)
\end{align*}

where $s = s\left(\frac{y_{2}}{y_{1}},\frac{y_{3}}{y_{1}^{p}}\right)$
and 
$\left( \frac{y_{2}}{y_{1}}, \frac{y_{3}}{y_{1}^{p}}\right) \in \operatorname{int}(\Omega_{3})\setminus \operatorname{Ang}(s_{0})$.
\end{Le}
 As it was mentioned in Remark~\ref{razryv}, the gradient function $t(s)$ is not necessarily differentiable at point $s_{0}$, this is the reason of the requirement $\left( \frac{y_{2}}{y_{1}}, \frac{y_{3}}{y_{1}^{p}}\right) \in  \operatorname{int}(\Omega_{3}) \setminus\operatorname{Ang}(s_{0})$ in the lemma. However, from the proof of the lemma, the reader can easily see that $N''_{y_{1}y_{1}}N''_{y_{3}y_{3}}-(N''_{y_{1}y_{3}})^{2}=0$ whenever the points $\left(\frac{y_{2}}{y_{1}}, \frac{y_{3}}{y_{1}^{p}}\right)$ belong to the interior of the domain  $\operatorname{Ang}(s_{0})$.
\begin{proof}
Definition of the candidate $N$ (see (\ref{extendab})) implies  $N''_{y_{3}y_{3}} = t'_{2}(s) s'_{y_{3}}$, $N''_{y_{3}y_{1}}=t'_{2} s'_{y_{1}}$, 

\begin{align}\label{der1}
N'_{y_{1}} = y^{p-1}_{1}\left(p\Bell \left(\frac{y_{2}}{y_{1}},\frac{y_{3}}{y_{1}^{p}}\right) - t_{1}\frac{y_{2}}{y_{1}} -pt_{2}\frac{y_{3}}{y^{p}_{1}}\right).
\end{align}
Condition (\ref{bellmanfunction}) implies

\begin{align*}
\Bell\left(\frac{y_{2}}{y_{1}},\frac{y_{3}}{y_{1}^{p}}\right)=
f(s)  + t_{1}\cdot \left( \frac{y_{2}}{y_{1}} - s\right) + t_{2} \cdot \left(\frac{y_{3}}{y_{1}^{p}} - g(s)  \right).
\end{align*}
We substitute this expression for $\Bell \left(\frac{y_{2}}{y_{1}},\frac{y_{3}}{y_{1}^{p}}\right)$ into (\ref{der1}), and  we obtain:

\begin{align}\label{erti}
N'_{y_{1}} = y^{p-1}_{1}\left( pf + \frac{y_{2}}{y_{1}}t_{1}(p-1) -pst_{1} - pgt_{2}\right).
\end{align}

Condition $\left( \frac{y_{2}}{y_{1}}, \frac{y_{3}}{y_{1}^{p}}\right) \in \operatorname{int}(\Omega_{3})\setminus \operatorname{Ang}(s_{0})$ implies the equality  $N''_{y_{1}y_{3}} = N''_{y_{3}y_{1}}$ which in turn gives

\begin{align*}
t'_{2} s'_{y_{1}}=y^{p-1}_{1}\left( pf' + \frac{y_{2}}{y_{1}}t'_{1}(p-1) -(pst_{1} + pgt_{2})'_{s}\right)s'_{y_{3}}.
\end{align*}
Hence 
\begin{align}\label{support2}
t'_{2}\cdot  (s'_{y_{1}})^{2}=y^{p-1}_{1}\left( pf' + \frac{y_{2}}{y_{1}}t'_{1}(p-1) -(pst_{1} + pgt_{2})'_{s}\right)s'_{y_{3}}s'_{y_{1}}.
\end{align}

We keep in mind this identity and  continue our calculations 

\begin{align*} 
N''_{y_{1}y_{1}} = (p-1)y^{p-2}_{1}\left( pf + \frac{y_{2}}{y_{1}}t_{1}(p-2) -pst_{1} - pgt_{2}\right) + y^{p-1}_{1}\left( pf' + \frac{y_{2}}{y_{1}}t'_{1}(p-1) -(pst_{1} + pgt_{2})'_{s}\right)s'_{y_{1}}. 
\end{align*}

So, finally we obtain
\begin{align*}
N''_{y_{1}y_{1}}N''_{y_{3}y_{3}}-(N''_{y_{1}y_{3}})^{2} = 
t'_{2}\left( N''_{y_{1}y_{1}}s'_{y_{3}} - t'_{2} (s'_{y_{1}})^{2}\right).
\end{align*}

Now we use the identity (\ref{support2}), and we  substitute the expression  $t'_{2} (s'_{y_{1}})^{2}$:

\begin{align*} 
&N''_{y_{1}y_{1}}N''_{y_{3}y_{3}}-(N''_{y_{1}y_{3}})^{2} = t'_{2}s'_{y_{3}}\left( N''_{y_{1}y_{1}} - y^{p-1}_{1}\left( pf' + \frac{y_{2}}{y_{1}}t'_{1}(p-1) -(pst_{1} + pgt_{2})'_{s}\right)s'_{y_{1}}\right)=t'_{2}s'_{y_{3}} \times \\
&(p-1)y^{p-2}_{1}\left( pf + \frac{y_{2}}{y_{1}}t_{1}(p-2) -pst_{1} - pgt_{2}\right)  =-t'_{2}s'_{y_{3}}p(p-1)y_{1}^{p-2}\left(st_{1}+gt_{2}-f + \frac{y_{2}}{y_{1}}t_{1}\cdot \left(\frac{2}{p}-1\right) \right).
\end{align*}
\end{proof}

Now we are going to consider several cases when the points $(y_{2}/y_{1},y_{3}/y_{1}^{p})$ belong to the different subdomains in $\Omega_{3}$. Note that we always have $N''_{y_{3}y_{3}} \leq 0$, because of the fact that $\Bell$ is concave in $\Omega_{3}$ and  (\ref{extendab}). So we only have to check that the determinant of the Hessian  $N$ is negative. If the determinant of the Hessian is zero, then it is sufficient to ensure that $N''_{y_{3}y_{3}}$ is strictly negative, and if $N''_{y_{3}y_{3}}$ is also zero, then we need to ensure that $N''_{y_{1},y_{1}}$ is nonpositive.

\textbf{Domain $\Omega(\Theta[s_{0},y_{p}])$}.

In this case we can use the equality (\ref{neim3}), and we obtain that 
$$
st_{1}+gt_{2}-f = y_{p}t_{1}.
$$
Therefore 
\begin{align*}
N''_{y_{1}y_{1}}N''_{y_{3}y_{3}}-(N''_{y_{1}y_{3}})^{2} = 
-t'_{2}s'_{y_{3}}p(p-1)y_{1}^{p-2}t_{1}y_{p}\left(1 + \frac{y_{2}}{y_{1}}\right) \geq 0.
\end{align*}
 because $t_{1} \geq 0$. Indeed,  $t_{1}(s)$ is continuous on $[c,1]$, where $c$ is the root of the cup and  $\Bell''_{y_{2}y_{2}}=t'_{1}s'_{y_{2}} \leq 0$, therefore, because of the fact $s'_{y_{2}}>0$,  it suffices to check that $t_{1}(1) \geq 0$ which follows from the following inequality  
$$
t_{1}(1) = f'(1) - t_{2}(1)g'(1) = f'(1) >0.
$$

\textbf{Domain of linearity $\operatorname{Ang}(s_{0})$.}

This is the domain which is obtained by the triangle $ABC$, where $A=(-1,g(-1)),$ $B=(s_{0},g(s_{0})),$ and $C=(-1,h(s_{0}))$ if $s_{0}< y_{p}$ and by the infinity domain of linearity, which is  rectangular type, and which lies  between the chords $AB$, $BC',$ where $C'=(s_{0}, +\infty)$ and $AC'',$ where $C''=(-1,+\infty)$ (see Figure \ref{fig:vvv}).

Suppose the points $\left(y_{2}/y_{1}, y_{3}/y_{1}^{p} \right)$ belong to the interior of $\operatorname{Ang}(s_{0})$. Then the gradient function $t(s)$ of $\Bell$ is constant, and moreover $s\left( \frac{y_{2}}{y_{1}}, \frac{y_{3}}{y_{1}^{p}}\right)$ is constant. 
The fact that the determinant of the Hessian is zero in the domain of linearity (note that $s'_{y_{3}} =0$) implies that we only need to check $N''_{y_{1}y_{1}}<0$. Equality (\ref{erti}) implies
\begin{align*}
&N''_{y_{1}y_{1}} = (p-1)y_{1}^{p-2}\left(pf+\frac{y_{2}}{y_{1}}t_{1}(p-2)-ps_{0}t_{1}-pgt_{2} \right) \leq (p-1)y_{1}^{p-2}(pf-ps_{0}t_{1}-pgt_{2}-t_{1}(p-2))= 0.
\end{align*}
The last equality follows from (\ref{neim3}). The above inequality turns into the equality if and only if  $\frac{y_{2}}{y_{1}} = s_{0}$, this is the boundary point of $\operatorname{Ang}(s_{0})$.  

\textbf{Domain of vertical segments.}

On the vertical segments determinant of the Hessian is zero (for example, because the vertical segment is vertical segment in all directions) and $\Bell''_{y_{3}y_{3}}=0$, therefore, we must check that $N''_{y_{1}y_{1}} \leq 0$.
We note that $s(y_{2},y_{3}) = y_{2},$ therefore, 

\begin{align*}
&N''_{y_{1}y_{1}} = y_{1}^{p-2}\times \left[ (p-1)\left(pf + st_{1}(p-2)-pst_{1}-pgt_{2}\right)-s\left(pf'-t'_{1}s -t_{1}p -pg't_{2}\right)\right].
\end{align*}
However, from (\ref{bdcondition}) we have $pf'-t_{1}p-pg't_{2}=0,$ therefore, 
\begin{align*}
N''_{y_{1}y_{1}} = y_{1}^{p-2}\times \left[ (p-1)\left(pf - 2st_{1}-pgt_{2}\right)+s^{2}t'_{1}\right].
\end{align*}
Condition $t'_{1}\leq 0$ implies that it is sufficient  to show  $pf - 2st_{1}-pgt_{2} \leq 0$. 
We use (\ref{bdcondition}), and we find $t_{1} = f'-g't_{2}$. Hence, 
\begin{align*}
pf - 2st_{1}-pgt_{2} = pf-gpt_{2}-2s(f'-g't_{2}) = pf-2sf'-t_{2}(gp-2sg').
\end{align*}
Note that $gp-2sg' \geq 0$ (because $s\geq 0$ and $g'\leq 0$), and we recall that from (\ref{bdcondition}) and the fact that on the vertical segments $t_{2}$ is constant, since we have $\cos \theta(s) =0$ (see the expression of $t_{2}$ from Lemma~\ref{texnika}), so $t_{2}$ is constant and hence  $0 \geq t'_{1} = f'' - g''t_{2}$, therefore, we have $t_{2} \geq f''/g''$. Therefore,
\begin{align*}
pf-2sf'-t_{2}(gp-2sg') \leq pf - 2sf' - \frac{f''}{g''}(gp-2sg').
\end{align*}
Now we recall the values (\ref{bc}), (\ref{fc}), and after direct calculations we obtain 
\begin{align*}
pf - 2sf' - \frac{f''}{g''}(gp-2sg') = \frac{f (1-s^{2})p(p-2)(\tau^{2}(1+s)^{2}+(1-s)^{2}+2\tau^{2}(1-s^{2}))}{(p-1)((1+s)^{2}+\tau^{2}(1-s)^{2})^{2}} \leq 0.
\end{align*}

\textbf{Domain of the cup $\Omega(\Theta_{\operatorname{cup}}((c,s_{0}],g))$.}

\begin{wrapfigure}[16]{r}{0pt}
\includegraphics[scale=0.85]{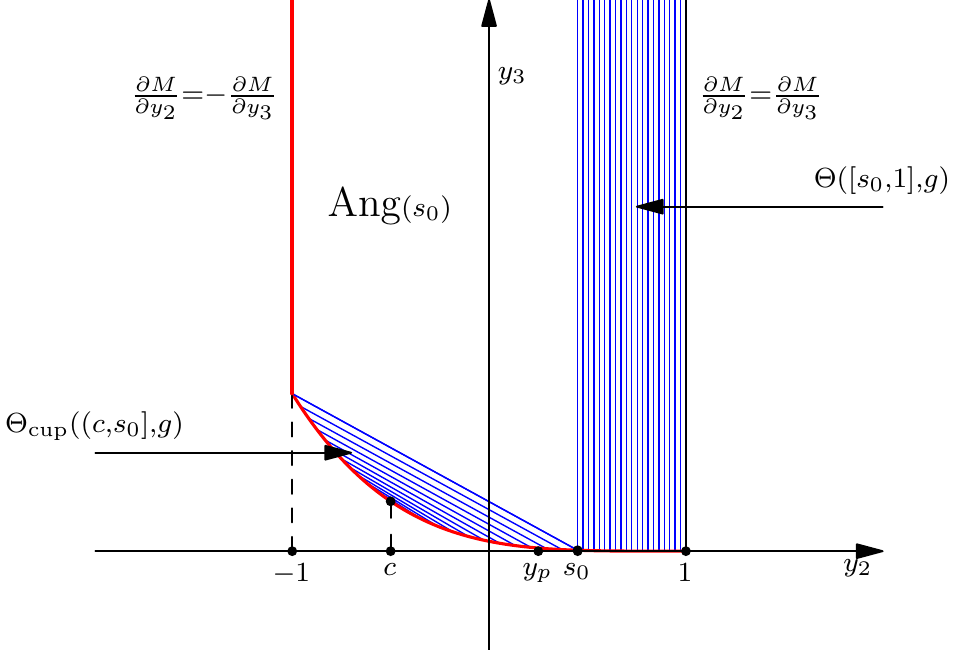}
\caption{Case $u\left(\frac{1}{p-1}\right)\geq 0$ }
\label{fig:vvv}
\end{wrapfigure}
The condition that $N''_{y_{3}y_{3}}$ is strictly negative in the cup implies that we  
only need to show $st_{2}+gt_{3}-f + \frac{y_{2}}{y_{1}}t_{1}(\frac{2}{p}-1) \geq 0$, where $s=s(y_{2}/y_{1}, y_{3}/y_{1}^{p})$ and the points $y = (y_{2}/y_{1}, y_{3}/y^{p}_{1})$ lie in the cup. We can think that $y_{2}/y_{1} \to y_{2}$ and $y_{3}/y_{1}^{p} \to y_{3}$ and $s(y_{2}/y_{1}, y_{3}/y_{1}^{p}) \to s(y_{2},y_{3})$, and we can think that the  points $(y_{2},y_{3})$ lie in the cup. 
Therefore it suffices to show that $st_{2}+gt_{3}-f + y_{2}t_{1}(\frac{2}{p}-1) \geq 0$, where $y=(y_{2},y_{3}) \in \Omega(\Theta_{\mathrm{cup}}((c,s_{0}],g))$. On a segment with the fixed endpoint $(s,g(s))$  the expressions $s,t_{1},g(s), t_{2},f(s)$ are constant, except of $y_{2}$, so the expression $st_{1}+gt_{2}-f + y_{2}t_{1}(\frac{2}{p}-1)$ is linear with respect to the $y_{2}$ on the each segment of the cup. Therefore, the worst case appears  when $y_{2} = a(s)$ ($a(s)$ $-$ is the left end (it is abscissa) of the given segment). This is true because $t_{1}\geq 0$ (as it was already  shown) and $(\frac{2}{p}-1) \geq 0$. So, as the result, we derive that it is sufficient to prove the inequality 
\begin{align}\label{support3}
st_{1}+gt_{2}-f + a(s)t_{1}\cdot \left(\frac{2}{p}-1\right)=t_{1}(s-a(s))+gt_{2}-f+\frac{2a(s)}{p}t_{1} \geq 0.
\end{align}
We use the identity (\ref{bellmanfunction}) at the point $y=(a(s),g(a(s)))$, and we find that 
$$
t_{1}(s-a(s))+gt_{2}-f = g(a(s))t_{2}-f(a(s)).
$$
We substitute this expression into (\ref{support3}) then we will get that it suffices  to prove the inequality:
\begin{align}\label{support4}
g(a(s))t_{2}-f(a(s))+\frac{2a(s)}{p}t_{1} \geq 0.
\end{align}

We differentiate condition $\Bell(a(s),g(a(s)))=f(s)$ with respect to $s$. Then we find the expression for $t_{1}(s)$, namely $t_{1}(s) = f'(a(s))-t_{2}(s)g'(a(s))$.
After substituting this expression into (\ref{support4}) we obtain that:

\begin{align*}
g(a(s))t_{2}-f(a(s))+\frac{2a(s)}{p}t_{1}  = \frac{1+z}{g'(z)}\left(\frac{(z-1)(\tau^{2}+1)f(z)}{((1+z)^{2}+\tau^{2}(1-z)^{2})g'(z)}  - t_{2}(s)\right),
\end{align*}
where $z=a(s)$. So it suffices to show that 
\begin{align}\label{sp4}
\frac{(z-1)(\tau^{2}+1)f(z)}{((1+z)^{2}+\tau^{2}(1-z)^{2})g'(z)}  - t_{2}(s) \leq 0
\end{align}
because $g'$ is negative. We are going to show that the condition (\ref{sp4}) is sufficient to check at the point $z=-1$.  Indeed, note that $(t_{2})'_{z}\geq 0$ on $[-1,c]$, where $c$ is the root of the cup, and also note that 
\begin{align*}
\left( \frac{(z-1)(\tau^{2}+1)f}{((1+z)^{2}+\tau^{2}(1-z)^{2})g'}\right)'_{z}=\frac{\tau^{2}+1}{p} (p-2)(1-z)^{-(p-1)}[(1+z)^{2}+\tau^{2}(1-z)^{2}]^{p/2 -2}2(1+z)\leq 0. 
\end{align*} 

The condition (\ref{sp4}) at the point $z=-1$ turns into the following condition
\begin{align*}
t_{2}(s_{0}) - \frac{\tau^{p-2}(\tau^{2}+1)}{p}\geq 0.
\end{align*}
Now we recall (\ref{differentiala}) and  $t_{2}(s_{0}) =(f'(-1)-f'(s_{0})/(g'(-1)-g'(s_{0}))$, therefore we have 
\begin{align*}
t_{2}(s_{0}) - \frac{\tau^{p-2}(\tau^{2}+1)}{p} \geq \frac{f''(-1)}{g''(-1)}-\frac{\tau^{p-2}(\tau^{2}+1)}{p} = \frac{\tau^{p}(p-1)^{2}+\tau^{p-2}}{p(p-1)}  >0. 
\end{align*} 

Thus we finish this section by the following remark.
\begin{Rem}
We still have to check the cases when the points $\left( \frac{y_{2}}{y_{1}}, \frac{y_{3}}{y_{1}^{p}}\right)$ belong to the boundary of $\operatorname{Ang}(s_{0})$ and the vertical rays $y_{2}=\pm 1$ in $\Omega_{3}$. The reader can easily see that in this case concavity of $N$  follows from the observation that $N \in C^{1}(\Omega_{1})$. Symmetry of $N$ covers the rest of the cases when  $\left(\frac{y_{2}}{y_{1}},\frac{y_{3}}{y_{1}^{p}}\right) \notin \Omega_{3}$. 
\end{Rem}
Thus we have constructed the candidate $N$. 

\section{Sharp constants via foliation}\label{sharp}
\subsection{Main theorem}
We remind the reader the definition of the functions $u(z)$, $g(s)$, $f(s)$ (see (\ref{lkonec}), (\ref{gc}), (\ref{fc})), the value $y_{p} = -1 + 2/p$ and the definition of the function $a(s)$ (see Lemma~\ref{functiona}, Lemma~\ref{unique} and Remark~\ref{extenda}). 
\begin{Th}\label{fullth}
Let $1 <p<2,$ and let  $G$ be the martingale transform of $F$ and let $|\mathbb{E}G| \leq \beta |\mathbb{E}F|$. Set $\beta' = \frac{\beta-1}{\beta+1}$.\\
\textup{(i)} If $u\left(\frac{1}{p-1}\right) \leq 0$ then
\begin{align}\label{putoloba}
\mathbb{E}(\tau^{2}F^{2}+G^{2})^{p/2} \leq \left(\tau^{2}+\max\left\{\beta,\frac{1}{p-1}\right\}^{2} \right)^{\frac{p}{2}}\mathbb{E}|F|^{p}.
\end{align}
\textup{(ii)}
If $u\left(\frac{1}{p-1}\right) > 0$ then
\begin{align*}
\mathbb{E}(\tau^{2}F^{2}+G^{2})^{p/2} \leq C(\beta')\mathbb{E}|F|^{p},
\end{align*}
where  $C(\beta')$ is continuous,  nondecreasing, and it is  defined by the following way: 
\begin{align*}
C(\beta')\df 
\displaystyle
\begin{cases} 
\left(\tau^{2}+\beta^{2} \right)^{p/2}, & \beta' \geq s^{*};\\
\\
\displaystyle
\frac{\tau^{p}}{1-\frac{2^{2-p}(1-s_{0})^{p-1}}{(\tau^{2}+1)(p-1)(1-s_{0})+2(2-p)}}, & \beta' \leq  -1+\frac{2}{p};\\
\\
\displaystyle
\frac{f'(s_{1})-f'(a(s_{1}))}{g'(s_{1})-g'(a(s_{1}))}, & R(s_{1},\beta')=0\quad \text{for} \quad s_{1}\in (\beta', s_{0});\\
\end{cases}
\end{align*}
where $s_{0} \in (-1+2/p,1)$ is the solution of the equation $u\left( \frac{1+s_{0}}{1-s_{0}}\right)=0,$ and the function $R(s,z)$ is defined as follows 
\begin{align*}
&R(s,z)\df -f(s)-\frac{f'(a(s))g'(s)-f'(s)g'(a(s))}{g'(s)-g'(a(s))}\left(z-s \right)+\\
&\frac{f'(s)-f'(a(s))}{g'(s)-g'(a(s))}g(s)=0, \quad z \in [-1+2/p,s^{*}],\quad  s\in [z,s_{0}].
\end{align*}
The value $s^{*} \in [-1+2/p, s_{0}]$ is the solution of the equation 
\begin{align}\label{zvezda}
\frac{f'(s^{*})-f'(a(s^{*}))}{g'(s^{*})-g'(a(s^{*}))}=\frac{f(s^{*})}{g(s^{*})}.
\end{align}
\end{Th}
\begin{proof}
Before we investigate some of the cases mentioned in the theorem, we should make the following observation.
The inequality of the type (\ref{putoloba}) can be restated as follows
\begin{align}\label{sharp1}
H(x_{1},x_{2},x_{3}) \leq C x_{3},
\end{align}
where  $H$ is defined by (\ref{bellmanf}) and $x_{1} = \mathbb{E}F,$ $x_{2} = \mathbb{E}G,$ $x_{3} = \mathbb{E}|F|^{p}$.  In order to derive the estimate (\ref{putoloba}) we have to find the sharp $C$ in (\ref{sharp1}). Because of the property (\ref{symmetry}) we can assume that both of the values  $x_{1},x_{2}$ are nonnegative. So non-negativity of $x_{1},x_{2}$ and the condition $|\mathbb{E}G| \leq \beta |\mathbb{E}F|$ can be reformulated as follows
\begin{align}\label{conditionvar}
-\frac{x_{1}+x_{2}}{2} \leq \frac{x_{2}-x_{1}}{2} \leq \left(\frac{\beta-1}{\beta+1}\right)\left(\frac{x_{1}+x_{2}}{2}\right).
\end{align}
The condition (\ref{conditionvar}) with (\ref{sharp1}) in terms  of the function $N$ and the variables $y_{1},y_{2},y_{3}$ means that we have to find the sharp $C$ such that 
\begin{align*}
N(y_{1},y_{2},y_{3}) \leq Cy_{3} \quad \text{for} \quad -y_{1} \leq y_{2} \leq \left(\frac{\beta-1}{\beta+1}\right)y_{1}, \quad \s{y} \in \Omega_{2}.
\end{align*}
Because of (\ref{extend}) the above condition can be reformulated as follows
\begin{align}\label{fcondition}
\Bell(y_{2},y_{3}) \leq Cy_{3} \quad \text{for} \quad -1 \leq y_{2} \leq \left(\frac{\beta-1}{\beta+1}\right), \quad y_{3} \geq g(y_{2}),
\end{align}
where $\Bell(y_{2},y_{3}) = N(1,y_{2}, y_{3})$. So our aim is to find the sharp $C$, in other words the value $\sup_{y_{1},y_{2}} \Bell/y_{3}$ where the supremum is taken from the domain mentioned in (\ref{fcondition}). Note that  the quantity $\Bell(y_{2},y_{3})/y_{3}$ increases with respect to the variable $y_{2}$. Indeed, ($\Bell(y_{2},y_{3})/y_{3})'_{y_{2}} = t_{1}(s(y))/y_{3}$, where the function $t_{1}(s)$ is nonnegative on $[c_{0},1]$ (see the end of the proof of the concavity condition in the domain $\Omega(\Theta[s_{0},y_{p}])$). Note that as we increase the value $y_{2}$ then the range of $y_{3}$  also increases. This means that the supremum of the expression $\Bell/y_{3}$ is  attained on the subdomain where  $y_{2} = (\beta-1)/(\beta+1)$. It is worth mentioning that since the quantity  $(\beta-1)/(\beta+1) \in [-1,1]$ increases as $\beta$ increases and because of the observation made above we see that the value $C$ increases as the $\beta'$ increases. 

\subsection{Case $y_{p}\leq s_{0}$.}
We are going to investigate the simple case (i). Remark~\ref{daxmareba} implies that 
 $s_{0}\leq y_{p}$, in other words, the foliation of vertical segments is  $\Theta([y_{p},1],g)$ where the value $\theta(s)$ on $[y_{p},1]$ is equal to $\pi/2$.
This means that $t_{2}(s)$ is constant on $[y_{p},1]$ (see Lemma~\ref{texnika}), and it is equal to $f(y_{p})/g(y_{p})=(\tau^{2}+\frac{1}{(p-1)^{2}})^{p/2}$ (see (\ref{fromsurf1})).

 If $\frac{\beta-1}{\beta+1}\geq y_{p}$, or equivalently $\beta \geq \frac{1}{p-1}$, then the function $\Bell$ on the vertical segment with the endpoint $(\beta',g(\beta'))$ where $\frac{\beta-1}{\beta+1} = \beta' \in [y_{p},1)$ has the expression (see (\ref{bellmanfunction}))
 \begin{align*}
 \Bell(\beta',y_{3})=f(\beta')+\frac{f(y_{p})}{g(y_{p})}(y_{3}-g(\beta')), \quad  y_{3} \geq g(\beta').
 \end{align*} 
Therefore, 
\begin{align}\label{ffcond}
\frac{\Bell(\beta',y_{3})}{y_{3}} =\frac{f(y_{p})}{g(y_{p})} + \frac{g(\beta')}{y_{3}}\left(\frac{f(\beta')}{g(\beta')}-\frac{f(y_{p})}{g(y_{p})}\right), \quad y_{3} \geq g(\beta').
\end{align}
The expression $f(s)/g(s)$ is strictly increasing on $(-1,1)$, therefore, the expression (\ref{ffcond}) attains its maximal value at the point $y_{3} = g(\beta')$. Thus,  we have 
\begin{align*}
\frac{\Bell(y_{2},y_{3})}{y_{3}} \leq  \frac{\Bell(\beta',y_{3})}{y_{3}} \leq \frac{\Bell(\beta',g(\beta'))}{g(\beta')} = \frac{f(\beta')}{g(\beta')} =\left(\tau^{2}+\beta^{2} \right)^{p/2} \quad \text{for} \quad -1\leq y_{2}\leq \beta',\; y_{3} \geq g(y_{2}).
\end{align*}

 If $\frac{\beta-1}{\beta+1}< y_{p}$, or equivalently $\beta < \frac{1}{p-1}$, then we can achieve such value for $C$ which was achieved at the moment $\beta = \frac{1}{p-1}$, and since the function $C = C(\beta')$ increases as $\beta'$ increases this value will be the best. Indeed, it suffices to look at the foliation (see Figure~\ref{fig:ccv}). For any fixed $y_{2}$ we send $y_{3}$ to $+\infty,$ and we obtain that 
 \begin{align*}
 &\lim_{y_{3} \to \infty}\frac{\Bell(y_{2},y_{3})}{y_{3}} =\lim_{y_{3} \to \infty} \frac{f(s)+t_{1}(s)(y_{2}-s)+t_{2}(s)(y_{3}-g(s))}{y_{3}} = \lim_{y_{3} \to \infty} t_{2}(s(y)) = t_{2}(y_{p}) = \left(\tau^{2} + \frac{1}{(p-1)^{2}}\right)^{p/2}.
 \end{align*}
\subsection{Case $y_{p}> s_{0}$.}
 As it was already mentioned, the condition in the case (ii) is equivalent to the inequality $s_{0}> y_{p}$ (see Remark~\ref{daxmareba}). This means that that the foliation of the vertical segments is $\Theta([s_{0},1],g)$ (see Figure~\ref{fig:vvv}).
 We know that  $C(\beta')$ is increasing. We remind that  we are going to maximize  the function $\frac{\Bell(y_{2},y_{3})}{y_{3}}$ on the domain mentioned in (\ref{fcondition}). It was already mentioned that we can require  $y_{2} = \left(\frac{\beta-1}{\beta+1} \right) = \beta'$. For such fixed $y_{2} = \beta' \in [-1,1]$ we are going to  investigate the monotonicity of the function $\frac{\Bell(\beta',y_{3})}{y_{3}}$. We consider several cases. Let $\beta' \geq s_{0}$. We differentiate the function $\Bell(\beta',y_{3})/y_{3}$ with respect to the variable $y_{3}$, and we use the expression (\ref{bellmanfunction}) for $\Bell$, so we obtain that  
 \begin{align*}
\frac{\partial }{\partial y_{3}}\left( \frac{\Bell(\beta',y_{3})}{y_{3}} \right) = \frac{t_{2}(\beta')y_{3} - \Bell(\beta',y_{3})}{y_{3}^{2}}=\frac{-f(\beta')+t_{2}(\beta')g(\beta')}{y_{3}^{2}}.
\end{align*}
Recall that $t_{2}(s) = t_{2}(s_{0})$ for $s \in [s_{0},1]$, therefore, direct calculations imply 
 
 \begin{align*}
t_{2}(\beta') = \frac{f(s_{0})-(s_{0}-y_{p})f'(s_{0})}{g(s_{0})-(s_{0}-y_{p})g'(s_{0})} < \frac{f(s_{0})}{g(s_{0})} \leq \frac{f(\beta')}{g(\beta')}, \quad \beta' \geq s_{0}.
\end{align*}
 This implies that 
  \begin{align*}
C(\beta') = \sup_{y_{3} \geq g(\beta')} \frac{\Bell(\beta',y_{3})}{y_{3}}  =\left. \frac{\Bell(\beta',y_{3})}{y_{3}}\right|_{y_{3}=g(\beta')}=\frac{f(\beta')}{g(\beta')} = (\tau^{2}+\beta^{2})^{p/2}.
 \end{align*}

 Now we consider the case $\beta' <s_{0}$. 
 
  For each point $y=(\beta',y_{3})$ that belongs to the line $y_{2}=\beta'$ there exists a segment $\ell(y) \in \Theta((c,s_{0}],g)$ with the endpoint $(s,g(s))$ where  $s \in [\max\{\beta',a(\beta')\},s_{0}]$. If the point $y$ belongs to the domain of linearity $\operatorname{Ang}(s_{0})$, then we can choose the value $s_{0}$, and consider any segment with the endpoints $y$ and $(s_{0},g(s_{0}))$ which surely  belongs to the domain of linearity. 
The reader can easily see that as we increase the value  $y_{3}$ the value $s$ increases as well.  So, 
\begin{align*}
\frac{\partial }{\partial y_{3}}\left( \frac{\Bell(\beta',y_{3})}{y_{3}} \right) = \frac{t_{2}(s)y_{3} - \Bell(\beta',y_{3})}{y_{3}^{2}}=\frac{-f(s)-t_{1}(s)(\beta'-s)+t_{2}(s)g(s)}{y_{3}^{2}}.
\end{align*}

 Our aim is to investigate the sign of the expression $-f(s)-t_{1}(s)(\beta'-s)+t_{2}(s)g(s)$ as we variate the value $y_{3} \in [g(\beta'), +\infty)$.
 Without loss of generality we can forget about the variable $y_{3}$, and we can variate only the value $s$ on the interval $[\max\{\alpha(\beta'),\beta'\},s_{0}]$.
 
 We consider the function $R(s,z) \df -f(s)-t_{1}(s)(z-s)+t_{2}(s)g(s)$ with the following domain  $-1 \leq z \leq s_{0}$ and $s \in [\max\{ \alpha(z),z\},s_{0}]$ (see Figure~\ref{fig:domainr}).
 As we already have seen $R(s_{0},s_{0}) <0$. Note that $R(s_{0},-1) >0$. Indeed, note that $R(s_{0},-1)=t_{2}(s_{0})g(-1)-f(-1)$. This equality follows from the fact that 
 \begin{align*}
 f(s_{0})-f(-1)=t_{1}(s_{0})(s_{0}+1)+t_{2}(s_{0})(g(s_{0})-g(-1)),
 \end{align*} 
 which is consequence of Lemma~\ref{bellmancup}. So, (\ref{fromcup}) and (\ref{differentiala}) imply 
 \begin{align*}
 t_{2}(s_{0}) = \frac{f'(-1)-f'(s_{0})}{g'(-1)-g'(s_{0})}  > \frac{f''(-1)}{g''(-1)} \geq \frac{f(-1)}{g(-1)}. 
 \end{align*}
 Note that the function $R(z,s_{0})$  is linear with respect to $z$. So on the interval $[-1,s_{0}]$ it has the root $y_{p}=-1+2/p$. Indeed, 
 \begin{align*}
 \frac{-f(s_{0})+t_{2}(s_{0})g(s_{0})+t_{1}(s_{0})s_{0}}{t_{1}(s_{0})}=y_{p}.
 \end{align*}
  The last equality follows from (\ref{fromcup}), (\ref{odno}) and (\ref{bdcondition}).
We need few more properties of the function $R(s,z)$.
 Note that for each fixed $z$, the function  $R(s,z)$ is nonincreasing on  $[\max\{\alpha(z),z\},s_{0}]$. Indeed
 \begin{align}\label{conc1}
 R'_{s}(s,z) = -f'(s)-t'_{1}(s)(z-s)+t_{1}(s)+t'_{2}(s)g(s)+t_{2}(s)g(s).
 \end{align}
We take into account the condition (\ref{bdcondition}), so the expression (\ref{conc1}) simplifies as follows
\begin{align*}
R'_{s}(s,z) = t'_{2}(s)g(s)+t'_{1}(s)(s-z). 
\end{align*}
We remind the reader equality  (\ref{extremal}) and the fact that $t'_{2}(s) \leq 0.$ Therefore we have $R'_{s}(s,z) = y_{3}t'_{2}(s)$ where $y_{3}=y_{3}(s) >0$.  Thus we see that $R(s,\beta') \geq 0$ for $\beta' \leq y_{p}$. 

So if the function $R(\cdot, z)$ at the right end on its domain $[\max\{\alpha(z),z\},s_{0}]$ is positive, this will mean that the function $\Bell/y_{3}$ is increasing, hence, the constant $C(\beta')$ will be equal to 
\begin{align*}
C(\beta') = \lim_{y_{3} \to \infty} \frac{\Bell(z,y_{3})}{y_{3}} = t_{2}(s_{0}) =  \frac{f'(-1)-f'(s_{0})}{g'(-1)-g'(s_{0})}
\end{align*}
(this follows from (\ref{fromcup}) and the structure of the foliation).  Since $u\left(\frac{1+s_{0}}{1-s_{0}}\right)=0$ and (\ref{coincide0}) direct computations show that 
 \begin{align}\label{davixrche}
 \frac{f'(-1)-f'(s_{0})}{g'(-1)-g'(s_{0})}=\frac{\tau^{p}}{1-\frac{2^{2-p}(1-s_{0})^{p-1}}{(\tau^{2}+1)(p-1)(1-s_{0})+2(2-p)}}. 
 \end{align}
 So it follows that if $\beta' \leq y_{p}$ then (\ref{davixrche}) is the value of $C(\beta')$. 
 
  If the function $R(\cdot, z)$ on the left end of its domain is nonpositive this will mean that the function $\Bell/y_{3}$ is decreasing, so the sharp constant will be the value of the function $\Bell(z,y_{3})/y_{3}$ at the left end of its domain
 \begin{align}\label{meore}
 C(\beta')=\left. \frac{\Bell(z,y_{3})}{y_{3}}\right|_{y_{3}=g(z)}=\frac{f(z)}{g(z)} = (\tau^{2}+\beta^{2})^{p/2}.
 \end{align}
 We recall that $c$ is the root of the cup and $c < y_{p}$  (see Remark~\ref{poryadok}). We will show that the function $R(z,s)$ is decreasing on the boundary $s=z$ for  $s \in (y_{p}, s_{0}]$. Indeed, (\ref{bdcondition}) implies
 \begin{align*}
 (R(s,s))'_{s} = -f'(s) +t'_{2}(s)g(s) +t_{2}(s)g'(s) = -t_{1}(s)+t'_{2}(s)g(s)<0.
 \end{align*}
The last inequality follows from the fact that $t'_{2}(s)\leq 0$ and $t_{1}(s) >0$ on $(c,1]$. Surely  $R(y_{p},y_{p}) > R(s_{0},y_{p})=0$, and we recall that $R(s_{0},s_{0})<0$, so there exists unique  $s^{*} \in [y_{p},s_{0}]$ such that  $R(s^{*},s^{*})=0$. This is equivalent to (\ref{zvezda}). So it is clear that $R(z,z) \leq 0$ for $z \in [s^{*},s_{0}]$. Therefore $C(\beta')$ has the value (\ref{meore}) for $\beta' \geq s^{*}$.  

The only case remains is when $\beta' \in [y_{p},s^{*}]$. We know that $R(z,z)\geq 0$ for $z\in [y_{p},s^{*}]$ and $R(s_{0},z)\leq 0$ for $z\in [y_{p},s^{*}]$. The fact that for each fixed $z$ the function $R(s,z)$ is decreasing implies the following: for each  $z\in [y_{p},s^{*}]$ there exists unique $s_{1}(z) \in [z,s_{0}]$ such that $R(z,s_{1}(z))=0$. Therefore, for $\beta' \in [y_{p},s^{*}]$ we have 
\begin{align}\label{maximalm}
C(\beta') = \frac{\Bell(\beta' ,y_{3}(s_{1}(\beta')))}{y_{3}(s_{1}(\beta'))},
\end{align}
where the value $s_{1}(\beta')$ is the root of the equation $R(s_{1}(\beta'),\beta')=0$. Recall that 
\begin{align}\label{roots1}
R(s_{1}(\beta'),\beta')&= t_{2}(s_{1})y_{3}(s_{1})-\Bell(\beta',y_{3}(s_{1}))=-f(s_{1})-t_{1}(s_{1})(\beta'-s_{1})+t_{2}(s_{1})g(s_{1}). 
\end{align} 
So the expression (\ref{maximalm}) takes the form 
\begin{align*}
C(\beta') = t_{2}(s_{1})=\frac{f'(s_{1}) - f'(a(s_{1}))}{g'(s_{1})-g'(a(s_{1}))}.
\end{align*}
 Finally, we remind the reader that 
\begin{align*}
t_{2}(s) &= \frac{f'(s)-f'(a(s))}{g'(s)-g'(a(s))},\\
t_{1}(s) &= \frac{f'(a(s))g'(s)-f'(s)g'(a(s))}{g'(s)-g'(a(s))}.
\end{align*} 
for $s \in (c,s_{0}]$, and we finish the proof of the theorem. 
\end{proof}

\section{Extremizers via foliation}\label{fol}
We set $\Psi(F,G) = \mathbb{E}(G^{2}+\tau^{2}F^{2})^{2/p}$.
Let $N$ be the candidate that we have constructed in Section~\ref{construction} (see~(\ref{extendab})). We define the candidate $B$ for the Bellman function $H$ (see (\ref{bellmanf}))  as follows
\begin{align*}
B(x_{1},x_{2},x_{3}) = N\left(\frac{x_{1}+x_{2}}{2},\frac{x_{2}-x_{1}}{2}, x_{3}\right), \quad (x_{1},x_{2},x_{3}) \in \Omega.
\end{align*}
We want to show that $B = H$. We already know that $B \geq H$ (see Lemma~\ref{mazhorantochka}). So, it remains to show that $B \leq H$.  We are going to do this as follows: for each point $\s{x} \in \Omega$ and any $\eps >0$ we are going to find an admissible pair  $(F,G)$ such that 
\begin{align}\label{conzero}
\Psi(F,G) > B(\s{x})-\varepsilon.
\end{align}
Up to the end of the current section we are going to work with the coordinates $(y_{1},y_{2},y_{3})$ (see (\ref{coorchange})). 
  It will be convenient for us to redefine the notion of  admissibility of the pair. 
\begin{Def}
We say that a pair $(F,G)$ is admissible for the point $(y_{1},y_{2},y_{3}) \in \Omega_{1}$, if $G$ is the martingale transform of $F$ and $\mathbb{E}(F,G,|F|^{p})=(y_{1}-y_{2},y_{1}+y_{2},y_{3})$.
\end{Def}
So in this case condition (\ref{conzero}) in terms of the function $N$ takes the following form: 
for any point $\s{y} \in \Omega_{1}$ and  for any $\eps >0$ we are going to find an admissible pair  $(F,G)$ for the point $\s{y}$ such that 
\begin{align}\label{conzero1}
\Psi(F,G) > N(\s{y})-\varepsilon.
\end{align}

\begin{wrapfigure}[17]{r}{0pt}
\includegraphics[scale=0.85]{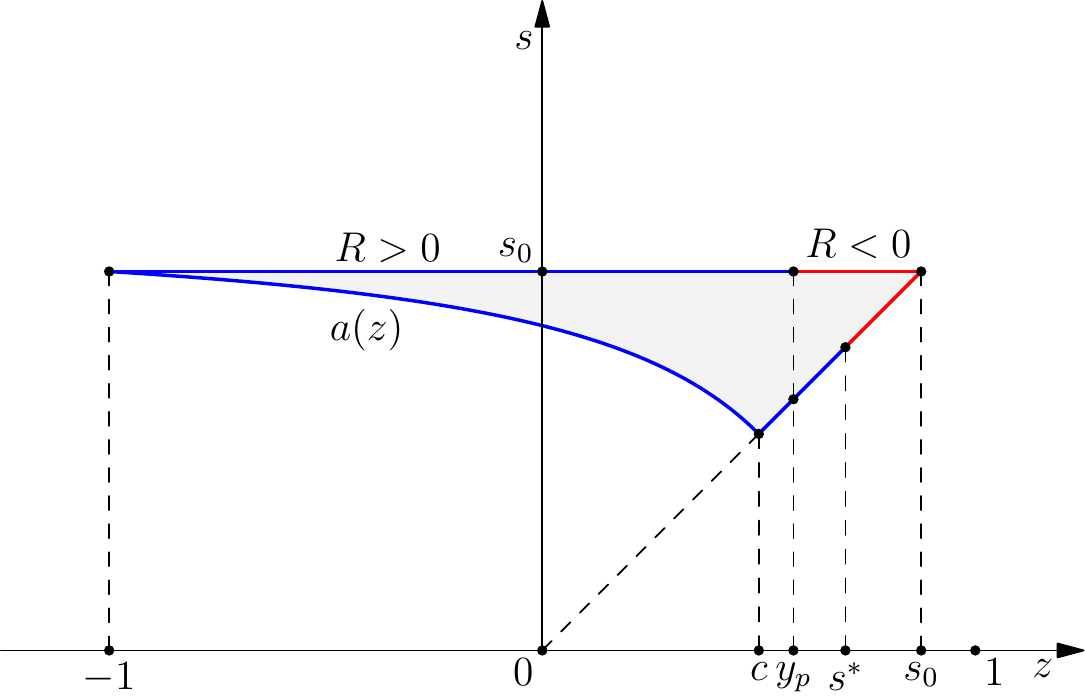}
\caption{Domain of $R(s,z)$}
\label{fig:domainr}
\end{wrapfigure}
We formulate the following obvious observations.  
\begin{Le}\label{PSI}
The following statements hold:
\begin{itemize}
\item[1.]
A pair $(F,G)$ is  admissible for the point $\s{y}=(y_{1},y_{2},y_{3})$ if and only if $(\tilde{F}, \tilde{G} ) =(\pm F,\mp G)$ is admissible for the point $\tilde{\s{y}} = (\mp y_{2},\mp y_{1},y_{3})$; moreover, 
$\Psi(\tilde{F}, \tilde{G}) = \Psi(F, G)$.
\item[2.]
A pair $(F,G)$ is admissible for the point $\s{y}=(y_{1},y_{2},y_{3})$, if and only if  $(\tilde{F}, \tilde{G} ) =(\lambda  F,\lambda  G)$ (where $\lambda \neq 0$) is admissible for the point $\tilde{\s{y}} = (\lambda y_{1},\lambda y_{2},|\lambda|^{p}y_{3})$; moreover, 
$\Psi(\tilde{F}, \tilde{G}) = |\lambda|^{p}\Psi(F, G)$.
\end{itemize}
\end{Le}

\begin{Def}
The pair of functions $(F,G)$ is called an {\em $\varepsilon$-extremizer}  for the point $\s{y} \in \Omega_{1}$ if $(F,G)$ is  admissible for the point $\s{y}$ and  $\Psi(F,G)> N(\s{y})-\varepsilon$.
\end{Def}

Lemma~\ref{PSI}, homogeneity, and  symmetry of $N$ imply that we only need to check (\ref{conzero1})  for the points $\s{y}\in \Omega_{1}$ where $y_{1}=1$ $(y_{2},y_{3})\in \Omega_{3}$. In other words, we show that 
$\Psi(F,G)> \Bell(y_{2},y_{3})-\varepsilon$ for some admissible $(F,G)$ for the point $(1,y_{2},y_{3})$ where $(y_{2},y_{3})\in \Omega_{3}$.  Further, instead of saying that $(F,G)$ is an admissible pair (or $\eps$-extremizer) for the point $(1,y_{2},y_{3})$ we just say that it is an admissible pair  (or an $\eps$-extremizer) for the point $(y_{2},y_{3})$. So we only have to construct $\eps$-extremizers in the domain $\Omega_{3}$. 

It is worth mentioning that we construct  $\eps$-extremizers $(F,G)$ such that $G$ will be the martingale transform of $F$ with respect to some filtration other than dyadic. A detailed explanation on how to pass from one filtration to another the reader can find in~\cite{SV}. 

We need a few  more observations.  For $\alpha \in (0,1)$ we define the {\em $\alpha-$ concatenation of the pairs  $(F,G)$ and $(\tilde{F}, \tilde{G})$} as follows
\begin{align*}
(F\bullet\tilde{F}, G\bullet \tilde{G})_{\alpha}(x) = 
\begin{cases}
(F,G)(x/\alpha ) &  x \in [0,\alpha],\\
(\tilde{F},\tilde{G})((x-\alpha)/(1-\alpha)) &  x \in [\alpha,1].
\end{cases}
\end{align*}
Clearly $\Psi((F\bullet\tilde{F}, G\bullet \tilde{G})_{\alpha}(x)) = \alpha \Psi(F,G)+(1-\alpha)\Psi(\tilde{F},\tilde{G})$. 
\begin{Def}
Any domain of the type $\Omega_{1}\cap \{ y_{1}=A \}$ where $A$ is some real number is said to be a {\em positive domain}. Any domain of the type $\Omega_{1} \cap \{ y_{2}=B\}$ where $B$ is some real number is said to be a {\em negative domain}.
\end{Def}

The following lemma is obvious.
\begin{Le}\label{marttr}
If $(F,G)$ is an admissible pair for a point $\s{y}$ and $(\tilde{F},\tilde{G})$ is an admissible pair for a point $\tilde{\s{y}}$ such that either of the following is true:  $\s{y}, \tilde{\s{y}}$ belong to a positive domain, or   $\s{y}, \tilde{\s{y}}$ belong to a negative domain, then $(F\bullet\tilde{F}, G\bullet \tilde{G})_{\alpha}$ is an admissible pair for the point $\alpha \s{y} + (1-\alpha) \tilde{\s{y}}$. 
\end{Le}

Let $(F,G)$ be an admissible pair for a point $\s{y}$, and let $(\tilde{F},\tilde{G})$ be an admissible pair for a point $\tilde{\s{y}}$. Let $\hat{\s{y}}$ be a point which belongs to the chord joining the points $\s{y}$ and $\tilde{\s{y}}$. 

\begin{Rem}
It is clear that if $(F^{+},G^{+})$ is admissible for a point $(y^{+}_{2},y^{+}_{3})$ and $(F^{-}, G^{-})$ is admissible for a point $(y^{-}_{2},y^{-}_{3})$ then an $\alpha-$ concatenation of these pairs is admissible for the point $(y_{2},y_{3}) = \alpha\cdot (y_{2}^{+},y^{+}_{3})+(1-\alpha)\cdot (y_{2}^{-},y_{3}^{-})$.
\end{Rem}

Now we are ready to construct $\eps$-extremizers in $\Omega_{3}$. The main idea is that these functions $\Psi$ and $\Bell$ are very similar: they obey almost the same properties. Moreover,  foliation plays crucial role in the contraction of $\eps-$ extremizers.  

\subsection{Case $s_{0}\leq y_{p}$.}

We want to find $\eps$-extremizers for the points in $\Omega_{3}$.

{\bf Extremizers in the  domain $\Omega(\Theta_{\mathrm{cup}}((c,s_{0}],g))$}.

Pick any  $y=(y_{2},y_{3}) \in \Omega(\Theta_{\mathrm{cup}}((c,s_{0}],g))$. Then there exists a segment $\ell(y) \in \Theta_{\mathrm{cup}}((c,s_{0}],g)$. Let $y^{+}=(s,g(s))$ and $y^{-}=(a(s),g(a(s))$ be the endpoints of $\ell(y)$ in $\Omega_{3}$.  We know $\eps$-extremizers at these points $y^{+},y^{-}$. Indeed, we can take the following $\eps$-extremizers $(F^{+},G^{+})=(1-s,1+s)$ and $(F^{-},G^{-})=(1-a(s),1+a(s))$ (i.e. constant functions). Consider an $\alpha-$concatenation $(F^{+}\bullet F^{-},G^{+}\bullet G^{-})_{\alpha}$, where $\alpha$ is chosen so that $y=\alpha y^{+}+(1-\alpha)y^{-}$.  We have  
\begin{align*}
&\Psi[(F^{+}\bullet F^{-},G^{+}\bullet G^{-})_{\alpha}] = \alpha \Psi(F^{+},G^{+})+(1-\alpha)\Psi(F^{-1},G^{-})>\\
& \alpha \Bell(y^{+}) + (1-\alpha) \Bell(y^{-})-\eps =\Bell(y) -\eps.
\end{align*}
The last equality follows from the linearity of $\Bell$ on $\ell(y)$.

{\bf Extremizers on the vertical line $(-1,y_{3})$, $y_{3}\geq h(s_{0})$}.

Now we are going to find $\eps$-extremizers for the points $(-1,y_{3})$ where $y_{3}\geq h(s_{0})$.  We use a similar idea mentioned in~\cite{VaVo1} (see proof of Lemma~3).
We define the functions $(F,G)$ recursively:

\begin{align*}
G(t)=
\begin{cases} 
-w &  0 \leq t < \eps;\\
\gamma \cdot g \left(\frac{t-\eps}{1-2\eps}\right)& \eps \leq t \leq 1-\eps;\\
w & 1-\eps < t \leq 1;\\
\end{cases}
\end{align*}

\begin{align*}
F(t)= 
\begin{cases} 
d_{-} &  0 \leq t < \eps;\\
\gamma \cdot f \left(\frac{t-\eps}{1-2\eps}\right)& \eps \leq t \leq 1-\eps;\\
d_{+} & 1-\eps < t \leq 1;\\
\end{cases}
\end{align*}
 where the nonnegative constants $w,d_{-},d_{+},\gamma$ will be obtained from the requirement $\mathbb{E}(F,G,|F|^{p})=(2,0,y_{3})$ and the fact that $G$ is the martingale transform of $F$. Surely $\av{G}{[0,1]}=0$. Condition $\av{F}{[0,1]}=2$ means that 
 \begin{align}\label{summa}
 (d_{-}+d_{+})\eps + 2\gamma (1-2\eps) = 2.
 \end{align}
 Condition $\av{|F|^{p}}{[0,1]}=y_{3}$ implies that 
 \begin{align}\label{summalp}
 y_{3} = \frac{\eps (d_{+}^{p}+d_{-}^{p})}{1-(1-2\eps)\gamma^{p}}.
 \end{align}
 Now we use the condition $|F_{0}-F_{1}|=|G_{0}-G_{1}|$. In the first step we split the interval $[0,1]$ at the point $\eps$ with the requirement $F_{0}-F_{1}=G_{0}-G_{1}$, from which obtain  $w=2-d_{-}$. In the second step we split at the point $1-\eps$ with the requirement $F_{1}-F_{2}=G_{2}-G_{1}$, obtaining $w=2\gamma-d_{+}$. From these two conditions we obtain $d_{-}+d_{+}=2(1+\gamma)-2w$, and by substituting in~(\ref{summa}) we find the $\gamma$
 \begin{align*}
 \gamma = 1+\frac{\eps w}{1-\eps}.
 \end{align*}
 Now we investigate what happens as $\eps$ tends to zero. Our aim will be to focus on the limit value $\lim_{\eps \to 0}w = w_{0}$.
  We have $1-(1-2\eps)\gamma^{p} \approx \eps(2-wp).$ So~(\ref{summalp}) becomes
 \begin{align}\label{konecxordy}
 y_{3} = \frac{\eps (d_{+}^{p}+d_{-}^{p})}{1-(1-2\eps)\gamma^{p}} \to \frac{2(2-w_{0})^{p}}{2-w_{0}p}.
 \end{align}
 Note that for $w_{0}=1+s$ equation (\ref{konecxordy}) is the same as (\ref{aes}). 
 By direct calculations we see that as $\eps \to 0$ we have 
 \begin{align*}
 \av{(G^{2}+\tau^{2}F^{2})^{p/2}}{[0,1]}=\frac{\eps[(w^{2}+\tau^{2}d_{-}^{2})^{p/2}+(w^{2}+\tau^{2}d_{+}^{2})^{p/2}]}{1-(1-2\eps)\gamma^{p}} \to \frac{2f(w_{0}-1)}{2-w_{0}p}.
 \end{align*}
 Now we are going to calculate the value $\Bell(-1,h(s))$ where $h(s)=y_{3}$. From (\ref{neim1}) we have 
 \begin{align*}
 \Bell(-1,h(s)) = h(s)t_{2}(s) - \frac{2}{p}t_{1}(s).
 \end{align*}
 By using  (\ref{bdcondition}) we express $t_{1}$ via $t_{2}$, also because of (\ref{aes}) and (\ref{fromsurf}) we have 
 \begin{align*}
 & \Bell(-1,h(s)) = h(s)t_{2}(s) - \frac{2}{p}t_{1}(s) = h(s)t_{2} - \frac{2}{p}(f'-t_{2}g')=\\
 &t_{2}(h(s)+\frac{2}{p}g')-f' \frac{2}{p}=\frac{f(s)-(s-y_{p})f'(s)}{g(s)-(s-y_{p})g'(s)}\left(\frac{2g}{p(y_{p}-s)}+\frac{2}{p}g'\right)-f'\frac{2}{p} = \frac{2}{p}\left[ \frac{f(s)}{y_{p}-s}\right] = \frac{2(2-w_{0})^{p}}{2-w_{0}p}.
 \end{align*} 
 Thus we obtain the desired result 
 \begin{align*}
 \av{(G^{2}+\tau^{2}F^{2})^{p/2}}{[0,1]} \to \Bell(-1,y_{3}) \quad \text{as}\quad  \eps \to 0.
 \end{align*}

{\bf Extremizers in the domain $\Omega(\Theta([s_{0},y_{p}),g)).$}

Pick any point $y=(y_{2},y_{3})\in \Omega(\Theta([s_{0},y_{p}],g))$. Then there exists a segment $\ell(y) \in \Theta([s_{0},y_{p}],g)$. Let $y^{+}$ and $y^{-}$ be the endpoints of this segment such that $y^{+}  = (-1,\tilde{y}_{3})$ for some $\tilde{y}_{3} \geq h(s_{0})$ and $y^{-} = (\tilde{s},g(\tilde{s}))$ for some $\tilde{s}\in [y_{p},s_{0})$. We remind the reader that we know $\eps$-extremizers for the points $(s,g(s))$ where  $s \in [s_{0},1]$, and we know $\eps$-extremizers  on the vertical line $(-1,y_{3})$ where $y_{3} \geq h(s_{0})$.  Therefore, as in the case of a cup, taking the appropriate $\alpha-$concatenation of these $\eps$-extremizers and using the fact that $\Bell$ is linear on $\ell(y)$,  we find an $\eps$-extremizer at point $y$. 

{\bf Extremizers in the domain $\operatorname{Ang}(s_{0}).$}

Pick any $y=(y_{1},y_{2}) \in \operatorname{Ang}(s_{0})$. There exist the points $y^{+} \in \ell^{+}$, $y^{-} \in \ell^{-}$, where $\ell^{+}=\ell^{+}(s_{0},g(s_{0})) \in \Theta([s_{0},y_{p}),g)$ and $\ell^{-}=\ell^{-}(s_{0},g(s_{0})) \in \Theta((c,s_{0}],g)$, such that $y=\alpha y^{+}+(1-\alpha) y^{-}$ for some $\alpha \in [0,1]$. We know $\eps$-extremizers at the points $y^{+}$ and $y^{-}$. Then by taking an $\alpha-$concatenation of these extremizers and using the linearity of $\Bell$ on $\operatorname{Ang}(s_{0})$ we can obtain an $\eps$-extremizer at the point $y$. 

{\bf Extremizers in the domain $\Omega(\Theta([y_{p},1],g)).$}
 
 Finally, we consider the domain of vertical segments $\Omega(\Theta[y_{p},1],g)$. Pick any point  $y=(y_{2},y_{3}) \in \Omega(\Theta[y_{p},1])$. Take an arbitrary point $y^{+}=(-1,y^{+}_{3})$ where $y_{3}^{+}$ is sufficiently large such that $y = \alpha y^{+}+(1-\alpha)y^{-}$ for some $\alpha \in (0,1)$ and some $y^{-}=(y_{2}^{-},y_{3}^{-})$ such that $(1,y^{-}_{2},y_{3}^{-}) \in \partial \Omega_{1}$. Surely, $y^{+}, y^{-}$ belong to a positive domain. Condition $(1,y^{-}_{2},y_{3}^{-}) \in \partial \Omega_{1}$ implies that we know an $\eps$-extremizer $(F^{-},G^{-})$ at the point $y^{-}$ (these are constant functions). We also know an $\eps$-extremizer at the point $y^{+}$. 
 Let $(F^{+}\bullet F^{-}, G^{+}\bullet G^{-})_{\alpha}$ be an $\alpha-$concatenation of these extremizers. Then 
 \begin{align*}
 \Psi[(F^{+}\bullet F^{-}, G^{+}\bullet G^{-})_{\alpha}] >\alpha \Bell(y^{+})+(1-\alpha) \Bell(y^{-})-\eps.
 \end{align*} 
 Note that the condition $y = \alpha y^{+}+(1-\alpha)y^{-}$ implies that 
\begin{align*}
\alpha = \frac{y_{3} - \frac{y_{2}}{y_{2}^{-}}y_{3}^{-}}{y_{3}^{+}+\frac{y_{3}^{-}}{y_{2}^{-}}}. 
\end{align*}
Recall that $\Bell(y_{2},g(y_{2}))=f(y_{2})$ and $\Bell(y^{+})=f(s)+t_{1}(s)(-1-s)+t_{2}(s)(y_{3}^{+}-g(s))$, where $s\in [s_{0},y_{p}]$ is such that a segment $\ell(s,g(s)) \in \Theta([s_{0},y_{p}),g)$ has an endpoint $y^{+}$.

Note that  as $y_{3}^{+} \to \infty$  all terms remain bounded; moreover, $\alpha \to 0,$ $y^{-} \to (y_{2},g(y_{2}))$ and $s \to y_{p}$. This means that 
\begin{align*}
&\lim_{y_{3}^{+}\to \infty } \alpha \Bell(y^{+})+(1-\alpha) \Bell(y^{-})-\eps = \\
&\lim_{y_{3}^{+}\to \infty }  t_{2}(s)y_{3}^{+} \left(\frac{y_{3} - \frac{y_{2}}{y_{2}^{-}}y_{3}^{-}}{y_{3}^{+}+\frac{y_{3}^{-}}{y_{2}^{-}}} \right) + f(y_{2})-\eps =t_{2}(y_{p})(y_{3} - g(y_{2}))+f(y_{2}) -\eps.
\end{align*}
We recall that $t_{2}(s) =t_{2}(y_{p})$ for $s \in [y_{p},1]$. Then  
\begin{align*}
\Bell(y) = f(y_{2})+t_{2}(s(y))(y_{3}-g(y_{2})) = f(y_{2})+t_{2}(y_{p})(y_{3}-g(y_{2})).
\end{align*}
Thus, if we choose $y_{3}^{+}$ sufficiently large then we can obtain a $2\eps$-extremizer for the point $y$. 
   
\subsection{Case $s_{0}> y_{p}$.}

In this case we have $s_{0} \geq y_{p}$ (see Figure~\ref{fig:vvv}). 
This case is a little bit more complicated than the previous one. Construction of $\eps$-extremizers  $(F,G)$ will be similar to the one presented in~\cite{RVV}.

We need a few more definitions. 

\begin{Def}
Let $(F,G)$ be an arbitrary pair of functions.  Let $(y_{2},g(y_{2})) \in \Omega_{3}$ and let $J$ be a subinterval of $[0,1]$. 
We define a new pair $(\tilde{F}, \tilde{G})$ as follows: 
\begin{align*}
(\tilde{F},\tilde{G})(x) = 
\begin{cases}
(F,G)(x) & x \in [0,1]\setminus J\\
(1-y_{2},1+y_{2}) & x \in J. 
\end{cases}
\end{align*}
We will refer to the new pair $(\tilde{F},\tilde{G})$ as \emph{ putting the constant $(y_{2},g(y_{2}))$ on the interval $J$ for the pair $(F,G)$}
\end{Def} 


It is worth mentioning that sometimes the new pair $(\tilde{F}, \tilde{G})$ we will denote by the same symbol $(F,G)$. 

\begin{Def}
We say that the pairs $(F_{\alpha}, G_{\alpha})$, $(F_{1-\alpha}, G_{1-\alpha})$ are obtained from the pair $(F,G)$ by splitting at the point $\alpha \in (0,1)$ if 
\begin{align*}
(F_{\alpha}, G_{\alpha}) = (F,G)(x\cdot \alpha ) \quad x \in [0,1];\\
(F_{1-\alpha}, G_{1-\alpha}) = (F,G)(x\cdot (1-\alpha)+\alpha) \quad x \in [0,1];
\end{align*}
\end{Def}
It is clear that $\Psi(F,G) = \alpha \Psi(F_{\alpha},G_{\alpha}) + (1-\alpha)\Psi(F_{1-\alpha}, G_{1-\alpha})$. Also note that if $(F_{\alpha}, G_{\alpha})$, $(F_{1-\alpha}, G_{1-\alpha})$ are obtained from the pair $(F,G)$ by splitting at the point $\alpha \in (0,1)$, then $(F,G)$ is an $\alpha-$concatenation of the pairs $(F_{\alpha}, G_{\alpha})$, $(F_{1-\alpha}, G_{1-\alpha})$. Thus, such operations as splitting and concatenation are opposite operations. 

Instead of explicitly presenting an admissible pair $(F,G)$ and showing that it is an $\eps$-extremizer, we present an algorithm which constructs the admissible pair, and we show that the result is an $\eps$-extremizer. 

 By the same explanations as in the case $s_{0}\leq y_{p}$, it is enough to construct an $\eps$-extremizer $(F,G)$ on the vertical line $y_{2}=-1$ of the domain $\Omega_{3}$. Moreover, linearity of $\Bell$ implies that  for any $A>0$, it is enough to construct $\eps$-extremizers for the points $(-1,y_{3})$, where  $y_{3} \geq A$.   Pick any point $(-1,y_{3})$ where $y_{3}=y_{3}^{(0)} > g(-1)$. Linearity of $\Bell$ on $\operatorname{Ang}(s_{0})$ and  direct calculations (see~(\ref{bellmanfunction}), (\ref{fromcup})) show that 
\begin{align}
\Bell(-1,y_{3}) = f(-1)+t_{3}(s_{0})(y_{3}-g(-1))= f(-1)+(y_{3}-g(-1))\frac{f'(-1)-f'(s_{0})}{g'(-1)-g'(s_{0})}. \label{value}
\end{align} 

We describe the first iteration. 
Let $(F,G)$ be an admissible pair for the point $(-1,y_{3})$, whose  explicit expression  will be described during the algorithm. For a pair $(F,G)$ we put a constant $(s_{0},g(s_{0}))$ on an interval $[0,\eps]$ where the value $\eps \in (0,1)$  will be given later. Thus we obtain a new pair $(F,G)$ which we denote by the same symbol. We want $(F,G)$ to be an admissible pair for the point $(-1,y_{3})$. Let the pairs $(F_{\eps},G_{\eps})$, $(F_{1-\eps},G_{1-\eps})$ be obtained from the pair $(F,G)$ by splitting at point $\eps$. It is clear  that $(F_{\eps},G_{\eps})$ is an admissible pair for the point $(s_{0},g(s_{0}))$. We want $(F_{1-\eps},G_{1-\eps})$ to be an admissible pair for the point $P=(\tilde{y}_{2},\tilde{y}_{3})$ so that 
\begin{align}\label{convhull}
&(-1,y_{3}) = \eps (s_{0},g(s_{0})) + (1-\eps) P.
\end{align}
Therefore we require
\begin{align}\label{certili}
P = \left(\frac{-1-\eps s_{0}}{1-\eps}, \frac{y_{3}-\eps g(s_{0})}{1-\eps} \right).
\end{align}
So we make the following simple observation: if $(F_{1-\eps}, G_{1-\eps})$ were an admissible pair for the point $P$, then $(F,G)$ (which is an $\eps-$concatenation of the pairs $(1-s_{0},1+s_{0})$ and $(F_{1-\eps}, G_{1-\eps})$) would be an admissible pair for the point $(-1,y_{3})$. Explanation of this observation is simple: note that these pairs $(F_{1-\eps}, G_{1-\eps})$ and $(1-s_{0},1+s_{0})$ are admissible pairs for the points $P$ and $(s_{0},g(s_{0}))$ which belong to a positive domain (see (\ref{convhull})); therefore,  the rest immediately follows from Lemma~\ref{marttr}. 
So we want to construct the admissible pair  $(F_{1-\eps}, G_{1-\eps})$ for the point (\ref{certili}).

We recall Lemma~\ref{PSI}  which implies that the pair $(F_{1-\eps}, G_{1-\eps})$ is admissible for the point $\left(1, \frac{-1-\eps s_{0}}{1-\eps}, \frac{y_{3}-\eps g(s_{0})}{1-\eps} \right)$  if and only if the pair $(\tilde{F},\tilde{G})$ where 
\begin{align*}
(F_{1-\eps},-G_{1-\eps}) =\frac{1+\eps s_{0}}{1-\eps}(\tilde{F},\tilde{G})
\end{align*}
is admissible for a point $W=\left(1,\frac{\eps-1}{1+\eps s_{0}}, \frac{(y_{3}-\eps g(s_{0}))}{(1+\eps s_{0})^{p}}\cdot  (1-\eps)^{p-1}\right).$
So, if we find the admissible pair $(\tilde{F},\tilde{G})$ then we automatically find the admissible pair $(F,G)$. 

 Choose  $\eps$ small enough so that  
$\left(\frac{\eps-1}{1+\eps s_{0}}, \frac{(y_{3}-\eps g(s_{0}))}{(1+\eps s_{0})^{p}}\cdot  (1-\eps)^{p-1}\right) \in \Omega_{3}$ and  
\begin{align*}
\left(\frac{\eps-1}{1+\eps s_{0}}, \frac{(y_{3}-\eps g(s_{0}))}{(1+\eps s_{0})^{p}}\cdot  (1-\eps)^{p-1}\right) = \delta (s_{0},g(s_{0})) + (1-\delta) (-1, y_{3}^{(1)} ) 
\end{align*}
for some $\delta \in (0,1)$ and $y_{3}^{(1)} \geq g(-1)$. 
Then 
\begin{align}
&\delta = \frac{\eps }{1+\eps s_{0}} = \eps + O(\eps^{2}) \nonumber\\
&y_{3}^{(1)}  = \frac{\frac{(y_{3}-\eps g(s_{0}))}{(1+\eps s_{0})^{p}}\cdot  (1-\eps)^{p-1}  - \frac{\eps}{1+\eps s_{0}} g(s_{0})}{1-\frac{\eps}{1+\eps s_{0}}} = y_{3}(1-\eps(p+ps_{0}-2))-2\eps g(s_{0}) + O(\eps^{2}).\label{approx}
\end{align}

For the pair $(\tilde{F},\tilde{G})$ we put a constant $(s_{0},g(s_{0}))$ on the interval $[0,\delta]$. We split the new pair $(\tilde{F},\tilde{G})$ at point $\delta$ so we get the pairs $(\tilde{F}_{\delta},\tilde{G}_{\delta})$ and $(\tilde{F}_{1-\delta},\tilde{G}_{1-\delta})$. We make a similar observation as above.  It is clear that if we know the admissible pair $(\tilde{F}_{1-\delta},\tilde{G}_{1-\delta})$  for the  point $(-1,y_{3}^{(1)})$ then we can obtain an admissible pair $(\tilde{F},\tilde{G})$ for the point $\left(\frac{\eps-1}{1+\eps s_{0}}, \frac{(y_{3}-\eps g(s_{0}))}{(1+\eps s_{0})^{p}}\cdot  (1-\eps)^{p-1}\right)$.  Surely $(\tilde{F},\tilde{G})$ is a $\delta-$concatenation of the pairs $(1-s_{0},1+s_{0})$ and $(\tilde{F}_{1-\delta},\tilde{G}_{1-\delta})$. 


We summarize the first iteration. We took $\eps \in(0,1)$, and we  started from the pair $(F^{(0)},G^{(0)})=(F,G)$, and  after one iteration  we came  to the pair $(F^{(1)},G^{(1)})=(\tilde{F}_{1-\delta},\tilde{G}_{1-\delta}).$ We showed that if $(F^{(1)},G^{(1)})$ is an admissible pair for the  point $(1,y_{3}^{(1)}),$ then the pair $(F^{(0)},G^{(0)})$ can be obtained from the pair  $(F^{(1)},G^{(1)})$; moreover, it is admissible for the point $(1,y_{3}^{(0)})$. 

Continuing these iterations, we obtain the sequence of numbers $\{ y_{3}^{(j)}\}_{j=0}^{N}$ and the sequence of pairs $\{(F^{(j)},G^{(j)}) \}_{j=0}^{N}$.  Let $N$ be such that $y_{3}^{(N)}\geq g(-1)$. It is clear that if $(F^{(N)},G^{(N)})$ is an admissible pair for the point $(-1,y_{3}^{(N)})$ then the pairs $\{(F^{(j)},G^{(j)}) \}_{j=0}^{N-1}$ can be determined uniquely, and, moreover, $(F^{(j)},G^{(j)})$ is an admissible pair for the  point $(-1,y_{3}^{(j)})$ for all $j=0,..,N-1$.
 
Note that  we can choose  sufficiently small $\eps\in (0,1)$,  and we can find  $N=N(\eps)$ such that $y_{3}^{(N)} = g(-1)$ (see (\ref{approx}), and recall that $s_{0} > y_{p}$).  In this case the admissible pair $(F^{(N)},G^{(N)})$ for the point $(-1,y_{3}^{(N)}) = (-1,g(-1))$ is a constant function, namely, $(F^{(N)},G^{(N)})=(2,0)$.  Now we try to find $N$ in terms of $\eps$, and we try to find the value of $\Psi(F^{(0)},G^{(0)})$. 

Condition (\ref{approx}) implies that $y_{3}^{(1)} =  y_{3}^{(0)}(1-\eps(p+ps_{0}-2))-2\eps g(s_{0}) +O(\eps^{2})$. We denote $\delta_{0}  = p+ps_{0}-2>0.$ Therefore, after the  $N$-th iteration we obtain
\begin{align*}
&y_{3}^{(N)} = (1-\eps \delta_{0})^{N}\left( y_{3}^{(0)}+ \frac{2 g(s_{0}) }{\delta_{0}} \right) -  \frac{2 g(s_{0}) }{\delta_{0}}+O(\eps).
\end{align*}
The requirement $y_{3}^{(N)}=g(-1)$ implies that 

\begin{align*}
(1-\eps \delta_{0})^{-N} =  \frac{y_{3}^{(0)}+\frac{2g(s_{0})}{\delta_{0}}}{g(-1)+\frac{2g(s_{0})}{\delta_{0}}}+O(\eps).
\end{align*}
This implies that $\limsup_{\eps \to 0} \eps \cdot N =  \limsup_{\eps \to 0} \eps \cdot N(\eps) <\infty$.
 Therefore, we get 
\begin{align}\label{znachenie}
e^{\eps \delta_{0} N} =  \frac{y_{3}^{(0)}+\frac{2g(s_{0})}{\delta_{0}}}{g(-1)+\frac{2g(s_{0})}{\delta_{0}}}+O(\eps).
\end{align}

Also note that
\begin{align*}
&\Psi(F^{(0)},G^{(0)})=\Psi(F,G) = \eps \Psi(F_{\eps},G_{\eps}) + (1-\eps) \Psi(F_{1-\eps},G_{1-\eps})=\\
&\eps f(s_{0}) + (1-\eps) \Psi(F_{1-\eps},G_{1-\eps})=\eps f(s_{0}) +
(1-\eps) \left(\frac{1+\eps s_{0}}{1-\eps}\right)^{p}\Psi (\tilde{F},\tilde{G})\\
&=\eps f(s_{0}) +(1-\eps)(1-\eps) \left(\frac{1+\eps s_{0}}{1-\eps}\right)^{p} \left[ \delta f(s_{0})+(1-\delta)\Psi(\tilde{F}_{1-\delta},\tilde{G}_{1-\delta}) \right]\\
&=2\eps f(s_{0}) + (1+\eps \delta_{0}) \Psi(F^{(1)},G^{(1)}) +O(\eps^{2}).
\end{align*}
Therefore, after the $N$-th iteration (and using the fact that $\Psi(F^{(N)},G^{(N)})=f(-1)$) we obtain
\begin{align}
&\Psi(F^{(0)},G^{(0)}) =  (1+\eps \delta_{0})^{N}\left(f(-1)+\frac{2f(s_{0})}{\delta_{0}}\right) - \frac{2f(s_{0})}{\delta_{0}} +O(\eps) =  \nonumber\\
&e^{\eps \delta_{0} N}\left(f(-1)+\frac{2f(s_{0})}{\delta_{0}}\right) - \frac{2f(s_{0})}{\delta_{0}}+O(\eps).\label{znachenie1}
\end{align}
The last equality follows from the fact that $\limsup_{\eps \to 0}\eps \cdot N(\eps) <\infty$.

Therefore (\ref{znachenie}) and (\ref{znachenie1}) imply that 
\begin{align*}
&\Psi(F^{(0)},G^{(0)}) = 
 \left( \frac{y_{3}^{(0)}+\frac{2g(s_{0})}{\delta_{0}}}{g(-1)+\frac{2g(s_{0})}{\delta_{0}}}\right)\left(f(-1)+\frac{2f(s_{0})}{\delta_{0}} \right) - \frac{2f(s_{0})}{\delta_{0}}+O(\eps)=\\
&f(-1)+(y_{3}^{(0)}-g(-1)) \left(\frac{  f(-1)+\frac{2f(s_{0})}{\delta_{0}} }{   g(-1)+\frac{2g(s_{0})}{\delta_{0}}}\right)+O(\eps).
\end{align*}

Now we recall (\ref{value}). So if we show that 
\begin{align}\label{bolo}
\frac{  f(-1)+\frac{2f(s_{0})}{\delta_{0}} }{   g(-1)+\frac{2g(s_{0})}{\delta_{0}}} = \frac{f'(-1)-f'(s_{0})}{g'(-1)-g'(s_{0})}
\end{align}
then (\ref{bolo}) will imply that $\Psi(F^{(0)},G^{(0)}) = \Bell(-1,y_{3}^{(0)}) +O(\eps)$. So choosing $\eps$ sufficiently small we can obtain the extremizer $(F^{(0)},G^{(0)})$ for the point $(-1,y_{3})$. Therefore, we need only to prove equality (\ref{bolo}). It will be convenient to make the following notations: set $f_{-}=f(-1)$, $f'_{-}=f'(-1)$, $f=f(s_{0})$, $f' = f'(s_{0})$, $g_{-}=g(-1)$, $g'_{-}=g'(-1)$, $g=g(s_{0})$ and $g' = g(s_{0})$. 
Then the equality (\ref{bolo}) turns into the following one
\begin{align}
\frac{\delta_{0}}{2} = \frac{fg'_{-}-fg'-f'_{-}g+f'g}{g'f_{-}-f'g_{-}}.
\end{align}
This simplifies into the following one
\begin{align*}
s_{0}-y_{p} = \frac{2}{p}\cdot \frac{fg'_{-}-fg'-f'_{-}g+f'g}{g'f_{-}-f'g_{-}} = \frac{fg'_{-}-fg'-f'_{-}g+f'g}{-g'f'_{-}+f'g'_{-}}
\end{align*}
which is true by (\ref{odno}). 
 
\section*{Acknowledgements}

I would like to express my deep gratitude to Professor A.~Volberg,  Professor V.~I.~Vasyunin, and Professor S.~V.~Kislyakov,  my research supervisors, for their patient guidance, enthusiastic encouragement, and useful critiques of this work. I would also like to thank A.~Reznikov, for his assistance in finding $\eps$-extremizers for the Bellman function.  
I would also like to extend my thanks to my colleagues and close friends P.~Zatitskiy, N.~Osipov, and D.~Stolyarov  for working together in Saint-Petersburg, and developing   a theory for minimal concave functions.

Finally, I wish to thank my parents for their support and encouragement throughout my study.

\end{document}